\documentclass[10pt,a4paper]{article}

\usepackage{amsmath,amssymb,epsfig}
\usepackage{epsfig}
\usepackage{graphicx}

\numberwithin{equation}{section}

\newtheorem{theorem}{Theorem}[section]
\newtheorem{lemma}[theorem]{Lemma}

\newtheorem{corollary}[theorem]{Corollary}

\newtheorem{definition}[theorem]{Definition}

\newtheorem{remark}[theorem]{Remark}

\newtheorem{facts}[theorem]{Facts}
\newtheorem{claim}[theorem]{Claim}

\newtheorem{observation}[theorem]{Observation}

\newenvironment{proof}{{\bf Proof.}}{}

\begin{document}
\bigskip

\centerline{\bf DISTRIBUTIVE MEREOTOPOLOGY:}

\smallskip

\centerline{\bf Extended  Distributive Contact Lattices}

\smallskip

\centerline{\bf Tatyana Ivanova and Dimiter Vakarelov}

\smallskip

\centerline {Department of Mathematical logic and its
applications,}

\centerline{ Faculty of Mathematics and Informatics, Sofia University}

 \centerline { e-mails:  tatyana.m.ivanova@gmail.com, dvak@fmi.uni-sofia.bg}

\bigskip

  \begin{abstract} The notion of contact algebra is one of the main tools in the region based theory of space. It is an extension of Boolean algebra with an additional relation  $C$ called contact. The elements of the Boolean algebra are considered as formal representations of spatial regions as analogs of physical bodies  and Boolean  operations are considered as operations for constructing new regions from given ones and also to define  some mereological relations between regions as \emph{part-of}, \emph{overlap} and  \emph{underlap}. The contact relation is one of the basic mereotopological relations between regions expressing  some  topological nature. It is used also to define some other  important mereotopological relations like  \emph{non-tangential inclusion},  \emph{dual contact}, \emph{external contact} and others. Most of these definitions are given by means of the operation of \emph{Boolean complementation}. There are, however, some problems related to the motivation of the operation of Boolean  complementation. In order to avoid these problems we propose a generalization of the notion of contact algebra by dropping the operation of complement and replacing the Boolean part of the definition by distributive lattice. First steps in this direction were made in \cite{dmvw1,dmvw2} presenting the notion of \emph{distributive contact lattice} based on contact relation as the only mereotopological relation. In this paper we   consider  as non-definable  primitives  the relations of contact, nontangential inclusion and dual contact, extending considerably the language of distributive contact lattices.  Part I of the paper is devoted to a suitable axiomatization of the new language called \emph{extended distributive contact lattice} (EDC-lattice) by means of universal first-order axioms  true in all contact algebras. EDC-lattices may be considered also as an algebraic tool for certain subarea of mereotopology,  called in this paper \emph{distributive mereotopology.}  The main result of Part I of the paper is a representation theorem, stating that each EDC-lattice can be isomorphically embedded into a contact algebra, showing in this way that the presented axiomatization preserves the meaning of mereotopological relations without considering Boolean complementation. Part II of the paper is devoted to  topological representation theory of EDC-lattices, transferring into the distributive case important  results from the topological representation theory of contact algebras. It is shown that under minor additional assumptions on distributive lattices as \emph{extensionality} of the definable relations of \emph{overlap} or \emph{underlap} one can preserve the good topological interpretations of regions as regular closed or regular open sets in topological space.

  \end{abstract}

\bigskip


%%%%%%%%%%%%%%%%%%%%%%%%%%%%%%%%%%%%%%
 \section{Introduction}
%%%%%%%%%%%%%%%%%%%%%%%%%%%%%%%%%%%%%%%%%%%%%%%%

In this paper we continue the research line started in  the  publications \cite{dmvw1,dmvw2}, devoted to  certain
non-classical approach to the region-based theory of space (RBTS), which roots goes back mainly to Whitehead \cite{W}. In
contrast to the classical Euclidean  approach, in which the notion of \emph{point}
is taken as one of the basic primitive notions in geometry and geometric figures are considered as sets of points, RBTS
 adopts as  primitives the more realistic spatial notion
of region (as an abstraction of  spatial or physical   body),
together with some basic relations and operations on regions. Some
of these relations come from mereology (see \cite{S}): e.g., \emph{part-of} ($x\leq
y$), \emph{overlap} ($xOy$), its dual \emph{underlap} ($x\widehat{O}y$), and some others
definable in terms of these. RBTS
extends classical mereology by considering some new relations
among regions which are topological in nature,  such as \emph{contact}
($xCy$), \emph{nontangential part-of} ($x\ll y$), \emph{dual contact} ($x\widehat{C}y$),  and some others
definable  by means of the contact and part-of relations.
This is one of the reasons that the extension of mereology with
these new relations is commonly called \emph{mereotopology}.  There is no clear difference in the literature between RBTS and mereotopology, and by some authors  RBTS  is related rather to the so called \emph{mereogeometry}, while mereotopology is considered only as a kind of point-free topology, considering mainly topological properties of things. In this paper we consider all these names almost as synonyms representing collections of various point-free theories of space. According to Whitehead the point-free approach to space should not disregard points  at all   - on the contrary, they are suitable high level abstractions which, as such,  should not be put on the base of the theory, but  have to be definable  by means of the other primitive notions of the theory. The Whitehead's criticism is based on the fact that points, as well as the other primitive notions in Euclidean geometry like lines and planes, do not have separate existence in reality, while for instance, spatial bodies as cubes, prisms, pyramids, balls, etc are things having analogs in reality.    In this sense the point-free approach to space  can be  considered as certain equivalent re-formulation of the classical point-based approach by means of  more realistic primitive notions.

Survey papers about  RBTS (and mereotopology) are \cite{V,BD, HG}  (see also the handbook \cite{A} and \cite{BTV} for some logics of space). Let us mention that in a sense   RBTS had been  reinvented in computer science, because of its more simple way of representing qualitative spatial information and in fact it initiated a special field in Knowledge Representation (KR) called \emph{Qualitative Spatial Representation and Reasoning} (QSRR). One of the most popular systems in QSRR is the   \emph{Region Connection Calculus} (RCC)  introduced in  \cite{RCC}. Note that RCC  influenced  various investigations in the field both of theoretical and applied nature. Survey papers about applications of RBTS and mereotopology in various applied areas are, for instance,  \cite{CR} and the book \cite{Hazarika}.

Let us note that one  of the main  algebraic tools in mereotopology is the notion of \emph{contact algebra}, which appears in the literature under different names and formulations  as extensions of Boolean algebra with some mereotopological relations  \cite{deVries,Stell,vdb,VDDB, BD,dw,dv,DuV}. The simplest system,  called just \emph{contact algebra} was introduced in \cite{dv} as  an extension of Boolean algebra $\underline{B}=(B, 0, 1, ., +, *)$ with a binary  relation $C$ called \emph{contact} and  satisfying several simple axioms:

\begin{quote}
(C1)  If $aCb$, then $a\not=0$ and $b\not=0$,

(C2) If $aCb$ and $a\leq a'$ and $b\leq b'$, then $a'Cb'$,

(C3)  If $aC(b+c)$, then $aCb$ or $aCc$,

(C4)    If $aCb$, then $bCa$,

(C5)   If $a.b\not=0$, then $aCb$.
\end{quote}

 The elements of the Boolean algebra are called   regions and the Boolean operations can be considered as some constructions of new regions by means of given ones. In this definition Boolean algebra stands for the mereological component, while the contact relation $C$  stands for the mereotopological component of the system. For instance the mereological relations  overlap $O$, underlap (dual overlap) $\widehat{O}$ and part-of $\leq$ have the following  definitions:  $aOb\leftrightarrow_{def} a.b\not=0$, $a\widehat{O}b\leftrightarrow_{def} a+b\not=1$ and  $\leq $ is just the lattice ordering. The unite element 1 is the region containing as its parts all regions, and the zero region $0$ symbolize the non-existing region and can be used to define the ontological predicate of existence: $a$ \textbf{exists } $\leftrightarrow_{def} a\not=0$. According to these definitions the axiom (C1) says that if $a$ and $b$ are in a contact then they exist, and axiom (C5) says that overlapping regions are in a contact.

By means of the contact relation one can define other mereotopological relations:  dual contact $a\widehat{C}b \leftrightarrow_{def} a^{*}Cb^{*}$, non-tangential part-of $a\ll b \leftrightarrow_{def} a\overline{C}b^{*}$, and some others.

Intuitively if we consider regions as certain sets of points, then contact $aCb$ means that $a$ and $b$ share a  common point, part-of  $a\leq b$ means that all points of $a$ are points of $b$, overlap $aOb$ means that $a$ and $b$ share an existing region (just $a.b\not=0$ is a part both of $a$ and of $b$), underlap $a\widehat{O}b$ means that there exists a non-universal region containing  both $a$ and $b$ (just $a+b\not=1$ contains both $a$ and $b$).

Let us note that standard  model  of Boolean algebra is the algebra of subsets of a given universe, so in such a model regions are pure sets and the mereological relations between regions are just the Boolean relations between sets. In this model one can not distinguish boundary and internal points of a given region and hence it can not  express all kinds of contact, for instance, the so called  \emph{external contact}  in which the contacting regions share only a boundary point (external contact is definable by the formula $aCb \land a\overline{O}b$). For this reason standard point models of contact algebras are of topological nature and consist of the Boolean algebras of \emph{regular closed sets}  in a given  topological space and the contact between two such sets means that they have a common point. Another topological model of contact algebra is the Boolean algebra of \emph{regular open sets} of a topological space, but in this model contact is not so intuitive and is definable by the formula: $aCb\leftrightarrow_{def} Cl(a)\cap Cl(b)\not=\varnothing$, where $Cl(a)$ is the topological closure operation. Let us mention that the topological representation theory of contact algebras  can be treated just as a realization of the Whitehead's idea of defining points and of recreation the point-based structure of the corresponding kind of space within a point-free system (see, for instance, the surveys \cite{V,BD}).

One of the motivations to put Boolean algebra on the base of the notion of contact algebra is based on the remark given  by Tarski (see for this \cite{S}, page 25) that one of the most popular mereological systems, namely the system of Lesniewski, can be identified with the complete Boolean algebra with zero deleted. If we are not interested in infinite unions and intersections then we can accept just Boolean algebra (with zero considered as \emph{non-existing region}, as mentioned above). In the papers \cite{dmvw1,dmvw2} a generalization of the notion of contact algebra is presented just by replacing the Boolean algebra by means of a (bounded) distributive lattice and obtaining in this way the notion of \emph{distributive contact lattice}. Some  motivations for this generalization are the following. First,  that Boolean algebra is a bounded distributive lattice and that the axioms of the contact relation do not use the operation of Boolean complementation $^{*}$ and have the same formulation in the language of bounded distributive lattice. Second, that the same can be said for the  basic mereological relations part-of, overlap and underlap - they have definitions in the language of distributive lattice without the operation of Boolean complement. Third, that the representation theory for distributive lattices is quite similar to the corresponding theory of Boolean algebras and we wanted to see if this can help us in  transferring the topological representation theory of contact algebras to the more general  theory of distributive contact lattices, keeping  the topological meaning of regions as regular closed sets. And finally, one philosophical motivation: the meaning of the Boolean complementation $a^*$ is not well motivated: if the region  $a$ represents a physical body, then what kind of body represents $a^*$? In the point-based models this is  "the rest out of  $a$" from the "whole space", the latter identified with the sum of all observed regions, the unit region 1. However, if we extend the area of our observation  we will obtain another unit, and then $a^*$ will be changed. But it is natural to assume that physical bodies should not depend on the area of observation in which they are included. As a result of this generalization, one can see that the paper \cite{dmvw2} generalizes almost all from the topological representation theory of contact algebras  developed for instance in \cite{dw,dv} and even more;  on the distributive case one can see some deep features which can not be observed in the Boolean case. For instance in the Boolean case mereological relations have some hidden properties which in the distributive case are not always fulfilled and have to be postulated explicitly (this is the so called  \emph{extensionality property} for the underlap and overlap relations). However, the obtained generalization in \cite{dmvw1,dmvw2} has some open problems. The mereotopological relations of non-tangential part-of and dual contact in contact algebras have definitions by means of the operation of complementation. However these relations have a  meaning in topological representation of contact algebras which does not depend on the operation of complementation on regular closed sets. Namely, if $a$ and $b$ are regular closed subsets of a topological space $X$, then $a\ll b$ iff $a\subseteq Int(b)$ and $a\widehat{C}b$ iff $Int(a)\cup Int(b)\not=X$, where $Int$ is the topological operation of interior of a set. Thus, it will be interesting to add these relations  as primitives  to the language of distributive contact lattices and to axiomatize them by means of a set of universal first-order axioms and then to extend the topological representation theory from \cite{dmvw2}.  This is one of the main open problems in \cite{dmvw2} which positive solution is subject of the present paper. One of the motivations for  this extension of the language of distributive contact lattice is that in this way we obtain a system with full duality: contact $C$ is dual to the dual contact $\widehat{C}$ and non-tangential part-of $\ll$ is dual to it converse $\gg$ and this symmetry makes possible to obtain proofs by duality.  The obtained new algebraic mereotopological system is named \emph{Extended Distributive Contact Lattice}, EDC-lattice for short. We will consider in the paper the topological representation theory of some axiomatic extensions of EDC-lattices with new  axioms yielding representations in better topological spaces, generalizing in this way the existing representation theory for contact algebras. Since all these investigations form a special subfield of mereotopology based on distributive lattices, we introduce for this subfield a special name  - \emph{distributive mereotopology}, which is included in the title of the present paper. Having in mind this terminology, then the subarea of mereotopology based on Boolean algebras should be named \emph{Boolean mereotopology}. Similar special names for other subfields of mereotopology depending on the corresponding mereological parts also can be suggested: for  instance the mereotopology considered in \cite{HWG,WHG1,WHG2} is based on some non-distributive lattices - hence \emph{non-distributive mereotopology}, and the mereotopological structures considered, for instance, in \cite{NV,F} are pure relational and without any algebraic lattice-structure in the set of regions - hence \emph{relational mereotopology}.

The paper is divided in two parts. Part I is devoted  to the axiomatization of the three mereotopological relations of contact $C$, dual contact $\widehat{C}$ and non-tangential part-of $\ll$ taken as primitives on the base of distributive lattice by means of universal first-order axioms, which remain true in contact algebras. The main result of this part is the abstract notion of \emph{Extended Distributive Contact Lattice} (EDC-lattice) and an embedding theorem of EDC-latices  into  contact algebras, showing in this way that the meaning of the contact, dual contact and non-tangential part-of relations is preserved in the language of EDC-lattices.  The method is based on a certain generalization of the Stone representation theory of distributive lattices \cite{Stone,bd}. As a consequence of the embedding theorem one can consider EDC-lattice also as the universal fragment of contact algbera based on the signature of distributive lattice and mereotopological relations of contact $C$, dual contact $\widehat{C}$ and non-tangential inclusion $\ll$. Relations of EDC-lattices with other mereotopological systems are also considered: EDC-lattices are \emph{relational mereotopological systems} in the sense of \cite{NV}, and the well known RCC-8 system of mereotopological relations is definable in the language of EDC-lattices.

Part II of the paper is devoted to the topological representation theory of EDC-lattices and some of their axiomatic extensions yielding representations in $T_{1}$ and  $T_{2}$ spaces. Special attention is given to dual dense and dense  representations (defined in Section \ref{embedding properties}) in contact algebras of regular closed and regular open subsets of  topological spaces. The method is an extension of the representation theory of distributive contact lattices \cite{dmvw2} and adaptation of some constructions from the representation theory of contact algebras \cite{dv,DuV}. In the concluding Section we discuss some open problems and future plans with applications in qualitative spatial representation and reasoning.

%%%%%%%%%%%%%%%%%%%%%%%%%%%%%%%%%%%%%%%%%%%%%%%%%%%%%%%%%%%%%%%%%%%%%%%%%%%%%%%%%%%%%%%

\newpage
\noindent {\bf PART I: EXTENDED DISTRIBUTIVE CONTACT LATTICES: \\AXIOMATIZATION AND EMBEDDING IN CONTACT ALGEBRAS}

\section{Extended distributive contact lattices.\\ Choosing the right axioms }
 \subsection{Contact algebras, distributive contact lattices and extended distributive contact lattices}

 As it was mention in the Introduction, \emph{contact algebra} is  a Boolean algebra $\underline{B}=(B,\leq,0,1,\cdot ,+,*, C)$ with an additional binary relation $C$ called \emph{contact}, and satisfying the following axioms:
 \begin{quote}
(C1)  If $aCb$, then $a\not=0$ and $b\not=0$,

(C2) If $aCb$ and $a\leq a'$ and $b\leq b'$, then $a'Cb'$,

(C3)  If $aC(b+c)$, then $aCb$ or $aCc$,

(C4)    If $aCb$, then $bCa$,

(C5)   If $a.b\not=0$, then $aCb$.
\end{quote}
Let us note that on the base of (C4) we have (C3') $(a+b)Cc$ implies $aCc$ or $bCc$.
\begin{remark}\label{remark1}{\rm Observe that the above axioms are universal first-order conditions on the language of Boolean algebra with the $C$-relation and not containing the Boolean complementation $^*$. This fact says that the axioms of $C$ will be true in any distributive sublattice of $\underline{B}$. $\square$ }
\end{remark}
 The Remark \ref{remark1} was one of the formal motivations for the definition of \emph{distributive contact lattice} introduced in \cite{dmvw1,dmvw2}: the definition  is obtained  just by replacing the underlying Boolean algebra by a bounded distributive lattice $(D,\leq,0,1,+,\cdot)$ and taking for the relation $C$ the same axioms. This makes possible to consider the main standard models of contact algebras, namely the algebras of regular closed or regular open  sets of a topological space, also as the main models for distributive contact lattices, just by ignoring the Boolean complementation $^*$ in this models. This was guaranteed  by Theorem 7 from  \cite{dmvw2} stating that every distributive contact lattice can be isomorphically embedded into a contact algebra, which fact  indicates also that the choice of the set of axioms for distributive contact lattice is sufficient for proving this theorem. Since our main goal in the present paper is to obtain a definition of distributive contact lattice extended with relations of dual contact $\widehat{C}$ and nontangential part-of $\ll$, we will follow here the above  strategy, namely to choose universal firs-order statements for the  relations $C, \widehat{C}, \ll$  as additional axioms which are true in arbitrary contact algebras and which guarantee the embedding into a contact algebra. The obtained algebraic system will be called \emph{extended distributive contact lattice}.  The next definition is a result of several preliminary experiments for fulfilling the above program.

\begin{definition}
\label{EDCL} {\bf Extended distributive contact lattice.} Let $\underline{D}=(D,\leq,0,1,+,\cdot,C,\widehat{C},\ll )$ be a bounded distributive lattice with three additional relations $C, \widehat{C}, \ll$, called respectively \textbf{contact}, \textbf{dual contact} and \textbf{nontangential part-of}. The obtained system, denoted shortly by $\underline{D}=(D,C, \widehat{C}, \ll)$, is called \textbf{extended distributive contact lattice} ( EDC-lattice, for short) if it satisfies the axioms listed below.

Notations: if $R$ is one of the relations $\leq, C, \widehat{C}, \ll$, then its complement is denoted by $\overline{R}$. We denote by $\geq$ the converse relation of $\leq$ and similarly  $\gg$ denotes the converse relation  of $\ll$.

\smallskip

 \textbf{Axioms for $C$ alone:} The axioms (C1)-(C5) mentioned above.

\smallskip

 \textbf{Axioms for $\widehat{C}$ alone:}

\begin{quote}
 $(\widehat{C}1)$ If $a\widehat{C}b$, then $a,b\not=1$,

 $(\widehat{C}2)$ If $a\widehat{C}b$ and $a'\leq a$ and $b'\leq b$, then $a'\widehat{C}b'$,

$(\widehat{C}3)$ If $a\widehat{C}(b\cdot c)$, then $a\widehat{C}b$ or $a\widehat{C}c$,

$(\widehat{C}4)$ If $a\widehat{C}b$, then $b\widehat{C}a$,

$(\widehat{C}5)$ If $a+b\not=1$, then $a\widehat{C}b$.
\end{quote}

\textbf{Axioms for $\ll$ alone:}

\begin{quote}

$(\ll 1)$ $0\ll 0$,

$(\ll 2)$ $1\ll 1$,

$(\ll 3)$ If $a\ll b$, then $a\leq b$,

$(\ll 4)$  If $a'\leq a\ll b\leq b'$, then $a'\ll b'$,

$(\ll 5)$  If $a\ll c$ and $b\ll c$, then $(a+b)\ll c$,

$(\ll 6)$  If $c\ll a$ and $c\ll b$, then $c\ll (a\cdot b)$,

$(\ll 7)$ If $a\ll b$ and $(b\cdot c)\ll d$ and $c\ll (a+d)$, then $c\ll d$.

\end{quote}

\textbf{Mixed axioms:}

\begin{quote}
$(MC1)$ If $aCb$ and $a\ll c$, then $aC(b\cdot c)$,

$(MC2)$ If $a\overline{C}(b\cdot c)$ and $aCb$ and $(a\cdot d)\overline{C}b$, then $d\widehat{C}c$,

$(M\widehat{C}1)$ If $a\widehat{C}b$ and $c\ll a$, then $a\widehat{C}(b+c)$,

$(M\widehat{C}2)$ If $a\overline{\widehat{C}}(b+c)$ and $a\widehat{C}b$ and $(a+d)\overline{\widehat{C}}b$, then $dCc$,

$(M\ll1)$ If $a\overline{\widehat{C}}b$ and $(a\cdot c)\ll b$, then $c\ll b$,

$(M\ll2)$ If $a\overline{C}b$ and $b\ll (a+c)$, then $b\ll c$.

\end{quote}
\end{definition}

 \begin{observation} {\bf Duality principle.}\label{duality}  {\rm  For the language of EDCL we can introduce the following principle of duality: dual pairs $(0,1), (\cdot, +), (\leq, \geq), (C,\widehat{C}), (\ll, \gg)$. By means of these pairs for each statement (definition) $A$ of the language we can define in an obvious way its  dual $\widehat{A}$. Then by a routine verification one can see that for each axiom $Ax$ from the list of axioms of EDCL its dual $\widehat{Ax}$ is also true.  On the base of this observation the proofs of dual statements will be omitted. Note, for instance, that each axiom from the first group (axioms for C alone) is dually equivalent to the corresponding axiom from the second group (axioms for $\widehat{C}$ alone) and vice versa, the third and fourth groups of axioms  (axioms for $\ll$ alone and mixed axioms) are closed under duality, for instance the axiom $(M \widehat{C}1)$ is dually equivalent to the axiom $(M C1)$, and $(M \ll2)$ is dually equivalent to $(M \ll1)$. $\square$}
 \end{observation}

%%%%%%%%%%%%%%%%%%%%%%%%%%%%%%%%%%%%%%%%%%%%%%%%%%%%%%%%%%%%%%%%%%%%%%%%%%%%%%%%%%%%%%%%%%%%%%%%%%%%%%%%%%%%%%%%%%%%%
\subsection{Relational models of EDC-lattices}\label{relmodels}

In order to prove that the axioms of EDC-lattices are true in contact algebras we will introduce a relational models of EDCL which are slight modifications of the relational models of contact algebras introduced in \cite{DuV} and called there \emph{discrette contact algebras}. The model is defined as follows.

Let $(W,R)$ be a relational system where $W$ is a nonempty set and $R$ is a reflexive and symmetric relation in $W$ and  let $a,b$ be arbitrary subsets of $W$. Define a contact relation between $a$ and $b$ as follows

(Def $C_{R}$)  $aC_{R}b$ iff $\exists x\in a$ and $\exists y\in b$ such that $x R y$.

\noindent Then any Boolean  algebra of subsets of $W$ with thus defined contact is a contact algebra, and moreover, every contact algebra is isomorphic to a contact algebra of such a kind \cite{DuV}.

We will modify this model for EDCL as follows: instead of Boolean algebras of sets we consider only families of subsets containing the empty set $\varnothing$ and the set $W$ and closed under the set-union and set-intersection which are bounded distributive lattices of sets. Hence we interpret lattice constants and operations as follows: $0=\varnothing$, $1=W$, $a\cdot b=a\cap b$, $a+b=a\cup b$. For the contact relation we preserve the definition (Def $C_{R}$). This modification is just a model of distributive contact lattice  studied in  \cite{dmvw2}.

Having in mind the definitions   $a\widehat{C}b\leftrightarrow_{def} a^{*}Cb^{*}$ and $a\ll b\leftrightarrow_{def} a\overline{C}b^{*}$) in Boolean algebras, we introduce the following definitions for $\widehat{C}$ and $\ll$ (for some convenience we present the definition of the negation of $\ll$):

\begin{quote}
(Def $\widehat{C}_{R}$ )  $a\widehat{C}_{R}b$ iff $\exists x\not\in a$ and $\exists y\not\in b$ such that $x R y$, and

(Def $\not\ll_{R}$) $a\not\ll_{R} b$ iff $\exists x\in a$ and $\exists y\not\in b$ such that $x R y$.

\end{quote}

\begin{lemma}\label{RelationEDCL} Let $(W,R)$ be a relational system with reflexive and symmetric relation $R$ and let $\underline{D}$ be any collection of subsets of $W$ which is a bounded distributive set-lattice with relations $C, \widehat{C}$ and $\ll$ defined as above. Then $(\underline{D},C_{R},\widehat{C}_{R},\\\ll_{R})$ is an EDC-lattice.
\end{lemma}

\begin{proof} Routine verification that all axioms of EDC-lattice are true. $\square $
\end{proof}

EDC-lattice $\underline{D}=(D,C_{R}, \widehat{C}_{R}, \ll_{R})$ over a relational system $(W,R)$ will be called \emph{discrete EDC-lattice}. If $D$ is a set of all subsets of $W$ then $\underline{D}$ is called a \emph{full discrete EDC-lattice}.

\begin{corollary}The axioms of the relations $C,\widehat{C}$ and $\ll$ are true in contact algebras.
\end{corollary}

\begin{proof} The proof follows by Lemma ref{RelationEDCL} and the fact that every contact algebra can be  isomorphically embedded into a discrete contact algebra over some relational system $(W,R)$ wit reflexive and symmetric relation $R$ \cite{DuV}. $\square$
\end{proof}
%%%%%%%%%%%%%%%%%%%%%%%%%%%
\section{Embedding  EDC-lattices into contact algebras}
%%%%%%%%%%%%%%%%%%%%%%%%%%%%%%%%%%%%%%%%%%%%%%%%%%%%%%%%%%%%%%%%%%%%%%%

The main aim of this  section is the proof a theorem stating that every EDC-lattice can be embedded into a full discrete EDC-lattice, which, of course is a Boolean contact algebra. As a consequence this will show that the axiomatization program for EDCL is fulfilled successfully. Since all axioms of EDC-lattice are universal first-order conditions, the axiomatization can be considered also as a characterization of the universal fragment of complement-free contact algebras based on the three relations.  We will use in the representation theory a Stone like technique developed in \cite{Stone} for the representation theory of distributive lattices.

%%%%%%%%%%%%%%%%%%%%%%%%%%%%%%%%%%%%%%%%%%%%%%%%%%%%%%%%%%%%%%%%%%%%%%%%%%%%%%%%%%%%%%%%%%%%%%%%%%%%%%%%%%%%%%
\subsection{Preliminary facts about filters and ideals in \\distributive lattices}
%%%%%%%%%%%%%%%%%%%%%%%%%%%%%%%%%%%%%%%%%%%%%%%%%%%%%%%%%%%%%%%%%%%%%%%%%%%%%%%%%%%%%%

We remaind some basic facts about filters and ideals in distributive lattices, for details see \cite{bd,Stone}.

Let $\underline{D}$ be a distributive lattice. A subset $F$ of $D$ is called a filter in $D$ if it satisfies the following conditions:
(f1) $1\in F$, (f2) if $a\in F$ and $a\leq b$ then $b\in F$, (f3) if $a,b\in F$ then $a.b\in F$. $F$ is a proper filter if $0\not\in F$, $F$ is  a prime filter if it is a proper filter  and $a+b\in F$ implies $a\in F$ or $b\in F$.

Dually, a subset $I$ of $D$ is an ideal if (i1) $0\in I$, (i2) if $a\in I$ and  $b\leq a$ then $b\in I$, (i3) if $a,b\in I$ then $a+b\in I$. I is a proper ideal if $1\not\in I$, $I$ is a prime ideal if it is a proper ideal and  $a.b\in I$ implies $a\in I$ or $b\in I$.

We will use later on some of the following facts without explicit mentioning.

\begin{facts}\label{facts} Let $\underline{D}$ be a bounded distributive lattice and Let $F, F_{1},F_{2}$ be filters and $I,I_{1},I_{2}$ be ideals.
\begin{enumerate}
\item  The complement of a prime filter is a prime ideal and  vice-versa.

\item $[a)=\{x\in D:a\leq x\}$ is the smallest filter containing $a$;

$(a]=\{x\in D: x\leq a \}$ is the smallest ideal containing $a$.

\item  $F_{1}\oplus F_{2}=\{c\in D: (\exists a\in F_{1}, b\in F_{2})(a\cdot b\leq c)\}=\{a\cdot b:a\in F_{1}, b\in F_{2}\}$ is the smallest filter containing $F_{1}$  and $F_{2}$.

 $[a)\oplus F=\{x\cdot y: a\leq x,\ y\in F\}$

 $I_{1}\oplus I_{2}=\{c\in D: (\exists a\in I_{1}, b\in I_{2})( c\leq a+b)\}=\{a + b:a\in I_{1}, b\in I_{2}\}$ is the smallest ideal containing $I_{1}$  and $I_{2}$.

 $(a]\oplus I=\{x+y: x\leq a,\ y\in I\}$.

 In both cases the operation $\oplus$ is associative and commutative.

 \item   $[a)\cap I=\varnothing$ iff $a\not\in I$

 If $(F\oplus [a))\cap I\not=\varnothing$ then $(\exists x\in F)(a\cdot x\in I)$,

 $(a]\cap F=\varnothing$ iff $a\not\in F$

 If $F\cap (I\oplus (a])\not=\varnothing$ then $(\exists x\in I)(a+x\in F)$.

\end{enumerate}

\end{facts}

The following three statements are well known in the representation theory of distributive lattices.

\begin{lemma}\label{feL}  Let $F_{0}$ be a filter, $I_{0}$ be an Ideal and $F_{0}\cap I_{0}=\varnothing$. Then:
\begin{enumerate}
\item \label{filter-extension Lemma} {\bf Filter-extension Lemma.}
There exists a prime filter $F$ such that $F_{0}\subseteq F$ and $F\cap I_{0}=\varnothing$.

\item \label{Ideal-extyension Lemma} {\bf Ideal-extension Lemma.} There exists a prime ideal $I$ such that $I_{0}\subseteq I$ and $F_{0}\cap I=\varnothing$.

\item \label{separation Lemma} {\bf Separation Lemma for filters and ideals.} There exist a (prime) filter $F$ and an (prime) ideal $I$ such that $F_{0}\subseteq F$, $I_{0}\subseteq I$, $F\cap I=\varnothing$, and $F\cup I=D$.

\end{enumerate}
\end{lemma}

\begin{remark}\label{Strong extension}{\rm Note that \emph{Filter-extension Lemma} is dual to the \emph{Ideal-extension Lemma} and that each of the three statement easily implies the other two. Normally they can be proved by application of the Zorn Lemma. The proof, for instance, of Filter-extension Lemma goes as follows. Apply the Zorn Lemma to the set $M=\{G: G $ is a filter, $F_{0}\subseteq G$ and $G\cap I_{0}=\varnothing\}$  and denote by $F$ one of its maximal elements. Then it can be  proved that $F$ is a prime filter, and this finishes the proof. The sketched proof gives, however, an additional property of the filter $F$, namely

 $(\forall x\not\in F)(\exists y\in F)(x\cdot y\in I_{0})$,

 \noindent which added to the formulation of the lemma makes it stronger. Since we will need later on this stronger version let us prove this property.

Suppose that $x\not\in F$ and consider the filter $F\oplus[x)$. Since $F$ is a maximal element of $M$, then $F\oplus[x)$ does not belong to $M$ and consequently $F\oplus[x)\cap I_{0}\not=\varnothing$. By the Fact \ref{facts}, 4, there exists $y\in F$ such that $x\cdot y\in I_{0}$. We formulate this new statement below as \emph{Strong filter-extension Lemma} and its dual as \emph{Strong ideal-extension Lemma}. We do not know if these two statements for distributive lattices are new, but we will use them in the representation theorem in the next section. $\square$ }

\end{remark}

\begin{lemma} \label{strong extension lemma} Let $F_{0}$ be a filter, $I_{0}$ be an Ideal and $F_{0}\cap I_{0}=\varnothing$. Then:
\begin{enumerate}
\item \label{strong filter-extension Lemma} {\bf Strong filter-extension Lemma.}
There exists a prime filter $F$ such that $F_{0}\subseteq F$ , $(\forall x\in F)(x\not\in I_{0})$ and $(\forall x\not\in F)(\exists y\in F)(x\cdot y\in I_{0})$.

\item \label{strong ideal-extyension Lemma} {\bf Strong ideal-extension Lemma.} There exists a prime ideal $I$ such that $I_{0}\subseteq I$, $(\forall x\in I)(x\not\in F_{0})$  and $(\forall x\not\in I)(\exists y\in I)(x + y\in F_{0})$.

\end{enumerate}
\end{lemma}

\subsection{Filters and Ideals in EDC-lattices}

In the next two lemmas we list some constructions of filters and ideals in EDCL which will be used in the representation theory of EDC-lattices.

\begin{lemma}
\label{FiltersIdeals1} Let
$\underline{D}=(D,C, \widehat{C}, \ll)$ be an EDC-lattice. Then:

\begin{enumerate}
\item  The set $I(x\overline{C}b)=\{x\in D: x\overline{C}b\}$ is an ideal,

\item the set $F(x\overline{\widehat{C}}b)=\{x\in D:x\overline{\widehat{C}}b\}$ is a filter,

\item  the set $I(x\ll b)=\{x\in D: x\ll b\}$ is an ideal,

\item the set $F(x\gg b)=\{x\in D:x\gg b\}$ is a filter.

\end{enumerate}
\end{lemma}

\begin{proof} \emph{1}. By axiom (C1) $0\overline{C}b$, so $0\in I(x\overline{C}b)$. Suppose $x\in I(x\overline{C}b)$ (hence $x\overline{C}b$) and $y\leq x $. Then by axiom (C2) $y\overline{C}b$). Let $x,y\in I(x\overline{C}b)$, hence $x\overline{C}b$ and $y\overline{C}b$. Then by axiom (C3) and (C4) we get $(x+y)\overline{C}b$ which shows that $x+y\in I(x\overline{C}b)$, which ends the proof of this case.

In a similar way one can proof \emph{3}. The cases \emph{2.} and \emph{4.} follow from \emph{1.}  and \emph{3.} respectively by duality. $\square$

\end{proof}

\begin{lemma}\label{FiltersIdeals2}
Let  $\underline{D}=(D,C, \widehat{C}, \ll)$ be an EDC-lattice and  Let $\Gamma$ be a prime filter in $\underline{D}$. Then:

\begin{enumerate}
\item  The set $I(x\overline{C}\Gamma)=\{x\in D: (\exists y\in \Gamma)(x\overline{C}y)\}$ is an ideal,

\item the set $F(x\overline{\widehat{C}}\overline{\Gamma})=\{x\in D:(\exists y\in \overline{\Gamma})(x\overline{\widehat{C}}y)\}$ is a filter,

\item  the set $I(x\ll \overline{\Gamma})=\{x\in D:(\exists y\in \overline{\Gamma})( x\ll y)\}$ is an ideal,

\item the set $F(x\gg\Gamma)=\{x\in D:(\exists y \in \Gamma)(x\gg y)\}$ is a filter.

\end{enumerate}
\end{lemma}

\begin{proof} Note that the Lemma remains true if we replace $\Gamma$ by a filter and $\overline{\Gamma}$ by an ideal.

\emph{1}. The proof that $I(x\overline{C}\Gamma)$ satisfies the conditions (i1) and (i2) from the definition of ideal is easy. For the condition (i3) suppose $x_{1},x_{2}\in I(x\overline{C}\Gamma)$. Then $\exists y_{1},y_{2}\in \Gamma$ such that $x_{1}\overline{C}y_{1}$ and $x_{2}\overline{C}y_{2}$, Since $\Gamma$ is a filter then $y=y_{1}\cdot y_{2}\in \Gamma$. Since $y\leq y_{1}$ and $y\leq y_{2}$, then by axiom (C2) we get $x_1\overline{C}y$ and $x_2\overline{C}y$. Then applying (C3') we obtain $(x_{1}+x_{2})\overline{C}y$, which shows that $x_{1}+x_{2}\in I(x\overline{C}\Gamma)$.

In a similar way one can prove \emph{3}. The proofs of \emph{2} and \emph{4} follow by duality from \emph{1} and \emph{3}, taking into account that $\overline{\Gamma}$ is an ideal. $\square$

\end{proof}

\subsection{Relational representation theorem for EDC-lattices}\label{relrepresentation}

Throughout this section we assume that $\underline{D}=(D,C,\widehat{C},\ll)$ is an EDC-lattice and let $PF(D)$ and $PI(D)$ denote the set of prime filters of $\underline{D}$ and the set of  prime ideals of $D$.  Let    $h(a)=\{\Gamma\in PF(D): a\in \Gamma\}$ be the well known Stone embedding mapping.  We shall construct a canonical relational structure $(W^{c},R^{c})$ related to $\underline{D}$ putting $W^{c}=PF(D)$ and defining $R^{c}$ for $\Gamma,\Delta\in PF(D)$ as follows:

$\Gamma R^{c}\Delta\leftrightarrow_{def} (\forall a,b\in D)(a\in \Gamma, b\in \Delta \rightarrow aCb)\&(a\not\in \Gamma, b\not \in \Delta\rightarrow a\widehat{C}b)\&(a\in \Gamma, b\not\in \Delta\rightarrow a\not\ll b)\&(a\not\in\Gamma, b\in \Delta \rightarrow b\not\ll a)$

For some technical reasons  and in order to use duality we introduce also the dual canonical structure $(\widehat{W}^{c}, \widehat{R}^{c})$ putting $\widehat{W}^{c}=PI(D)$ and for $\Gamma,\Delta\in PI(D)$, $\Gamma \widehat{R}^{c}\Delta \leftrightarrow_{def} \overline{\Gamma}R^{c}\overline{\Delta}$.

Our aim is to show that the Stone mapping $h$ is an embedding from $\underline{D}$ into the EDC-lattice over $(W^{c},R^{c})$ (see Section \ref{RelationEDCL}). First we need several technical lemmas.

\begin{lemma}
The canonical relations $R^{c}$ and $\widehat{R}^{c}$ are reflexive and symmetric.
\end{lemma}

\begin{proof} \textbf{( For $R^{c}$)}  Symmetry is obvious by the definition of $R^{c}$ and axioms (C4) and $(\widehat{C}4)$. In order to prove that $\Gamma R^{c}\Gamma$ suppose $a\in \Gamma$ and $b\in \Gamma$. Then $a\cdot b\in \Gamma$ and since $\Gamma$ is a prime filter, then $a.b\not=0$. Then by axiom (C5) we obtain $aCb$, which proves the first conjunct of the definition of $R^{c}$. For the second conjunct suppose that $a\not\in \Gamma$ and $b\not\in \Gamma$, then, since $\Gamma$ is a prime filter, $a+b\not\in \Gamma$ and hence $a+b\not=1$. Then by axiom $(\widehat{C}5)$ we get $a\widehat{C}b$. For the third conjunct suppose $a\in \Gamma$ and $b\not \in \Gamma$, which implies that $a\not\leq b$. Then by axiom $(\ll 3)$ we obtain $a\not \ll b$. The proof of the last conjunct is similar.

\textbf{(For $\widehat{R}^{c}$)} - by duality.
$\square$

\end{proof}

\begin{lemma}\label{C} (i) $aCb$ iff $(\exists\Gamma, \Delta\in PF(D))( a\in \Gamma$ and $b\in \Delta$ and $\Gamma R^{c}\Delta)$.

(ii)  $a\not\ll b$ iff $(\exists \Gamma, \Delta\in PF(D))(a\in \Gamma$ and $b\not\in \Delta$ and $\Gamma R^{c}\Delta)$.

\end{lemma}

\begin{proof} (i) Note that the proof is quite technical, so we will present it with full details. The reasons for this are twofold: first to help the reader to follow it more easily, and second, to skip the details in a similar proofs.

($\Leftarrow$) If $a\in \Gamma$ and $b\in \Delta$ then by the definition of $R^{c}$ we obtain  $aCb$.

($\Rightarrow$) Suppose $aCb$.\\
The proof will go on several steps.

\noindent {\bf  Step 1: construction of $\Gamma$.} Consider the ideal $I(x\overline{C}b)=\{x\in D: x\overline{C}b\}$ (Lemma \ref{FiltersIdeals1}). Since $aCb$, $a\not\in \{x\in D: x\overline{C}b\} $. Then $[a)\cap \{x\in D: x\overline{C}b\}=\varnothing$ and $[a)$ is a filter (see Facts \ref{facts}). By the Strong filter-extension lemma (see Lemma \ref{strong extension lemma}) there exists a prime filter $\Gamma$ such that  $[a)\subseteq \Gamma$ and $(\forall x\in \Gamma)(x\not\in \{x\in D: x\overline{C}b\}$ and    $(\forall x\not\in \Gamma)(\exists y\in \Gamma)(x\cdot y\in \{x\in D: x\overline{C}b\}$. From here we conclude that $\Gamma$ satisfies the following two properties:

\smallskip
($\#0$) $a\in \Gamma$,

\smallskip

($\#1$) If $x\in \Gamma$, then $xCb$, and

\smallskip

($\#2$) If $x\not\in \Gamma$, then there exists $y\in \Gamma$ such that $(x\cdot y)\overline{C} b$.

\smallskip
%%%%%%%%%%%%%%%%%%%%%%%%%%%%%%
\noindent {\bf  Step 2: construction of $\Delta$.} This will be done in two sub-steps.

{\bf  Step 2.1} Consider the filters and ideals definable by $\Gamma$ as in Lemma \ref{FiltersIdeals2}

 \noindent $I(x\overline{C}\Gamma)=\{x\in D: (\exists y\in \Gamma)(x\overline{C}y)\}$, $F(x\overline{\widehat{C}}\overline{\Gamma})=\{x\in D:(\exists y\in \overline{\Gamma})(x\overline{\widehat{C}}y)\}$, $I(x\ll \overline{\Gamma})=\{x\in D:(\exists y\in \overline{\Gamma})( x\ll y)\}$, and
$F(x\gg\Gamma)=\{x\in D:(\exists y \in \Gamma)(x\gg y\}$.
In order to apply the Separation Lemma we will prove the following condition:

$(\#3)$   $F(x\gg\Gamma)\oplus F(x\overline{\widehat{C}}\overline{\Gamma})\oplus [b)\cap I(x\overline{C}\Gamma)\oplus I(x\ll \overline{\Gamma})=\varnothing$.

  Suppose that $(\#3)$  is not true,  then for some $t\in D$ we have
%%%%%%%%%%%%%%%%%%%%%%%%%%%%%%%%%%

\noindent (1)   $t\in F(x\gg\Gamma)\oplus F(x\overline{\widehat{C}}\overline{\Gamma})\oplus [b)$ and

\noindent (2) $t\in I(x\overline{C}\Gamma)\oplus I(x\ll \overline{\Gamma})$.

It follows from (2) that $\exists k_1,k_2$ such that\\
(3) $k_1\in I(x\ll\overline{\Gamma})$ and\\
(4) $k_2\in I(x\overline{C}\Gamma)$ and\\
(5) $t=k_1+k_2$.

It follows from (1) that $\exists k_4,k_5,k_6\in D$ such that\\
(6) $k_4\in F(x\gg\Gamma)$ and\\
(7) $k_5\in F(x\overline{\widehat{C}}\overline{\Gamma})$ and\\
(8) $k_6\in[b)$ and\\
(9) $t=k_4\cdot k_5\cdot k_6$.

From (5) and (9) we get\\
(10) $k_1+k_2=k_4\cdot k_5\cdot k_6$.

It follows from (3), (4), (6) and (7)  that\\
(11) $\exists x_1\in\overline{\Gamma}$ such that $k_1\ll x_1$,\\
(12) $\exists x_2\in\Gamma$ such that $k_2\overline{C}x_2$,\\
(13) $\exists x_3\in \Gamma$ such that $x_3\ll k_4$,\\
(14) $\exists x_4\in\overline{\Gamma}$ such that $k_5\overline{\widehat{C}}x_4$.

Let $x=x_1+x_4$. Since $\overline{\Gamma}$ is an ideal, we obtain by (11) and (14) that\\
(15) $x\in\overline{\Gamma}$ and $x\not\in\Gamma$. Then by $(\#2)$ we get\\
(16) $\exists y\in\Gamma$ such that $(x\cdot y)\overline{C}b$.

Let $z=x_2\cdot x_3\cdot y$. Then by (12), (13) and (16) we obtain that\\
(17) $z\in\Gamma$\\
and by $(\#1)$ that\\
(18)$zCb$.

From $x_1\leq x$ and (11) by axiom $(\ll4)$ we get\\
(19) $k_1\ll x$.

From $x_4\leq x$ and (14) by axiom $(\widehat{C}2)$ we obtain\\
(20) $k_5\overline{\widehat{C}}x$.

From $z\leq x_2$ and (12) by axiom $(C2)$ we get\\
(21) $k_2\overline{C}z$.

From $z\leq x_3$ and (13) by axiom $(\ll4)$ we obtain\\
(22) $z\ll k_4$.

We shall show that the following holds\\
(23) $z\overline{C}(b\cdot k_1)$.

Suppose for the sake of contradiction that $zC(b\cdot k_1)$. From $b\cdot k_1\leq k_1$ and (19) by axiom $(\ll4)$ we get $(b\cdot k_1)\ll x$. From this fact and $zC(b\cdot k_1)$ by axiom $(MC1)$ we obtain $(b\cdot k_1)C(z\cdot x)$. But we also have $b\cdot k_1\leq b$, $z\cdot x\leq y\cdot x$, so by axiom $(C2)$ we get $bC(y\cdot x)$ - a contradiction with (16).

The following condition holds \\
(24) $z\overline{C}(b\cdot k_2)$.

To prove this suppose for the sake of contradiction  that $zC(b\cdot k_2)$. We also have $b\cdot k_2\leq k_2$, so by axiom $(C2)$ we get $zCk_2$ - a contradiction with (21).

Suppose that $zC(b\cdot(k_{1}+k_{2}))$. By axiom (C3) we have $zC(b\cdot k_{1})$ or $zC(b\cdot k_{2})$ - a contradiction with (23) and (24). Consequently $z\overline{C}(b\cdot(k_{1}+k_{2}))$ and by (10) we obtain $z\overline{C}(b\cdot k_{4}\cdot k_{5}\cdot k_{6})$. But $b\leq k_{6}$ (from (8)), so $b\cdot  k_{4}\cdot k_{5}\cdot k_{6}=b\cdot k_{4}\cdot k_{5}$. Consequently \\
(25) $z\overline{C}(b\cdot k_4\cdot k_5)$.

From (18) and (22) by axiom $(MC1)$ we get\\
(26) $zC(b\cdot k_4)$.

We shall show that the following condition holds\\
(27) $(z\cdot x)\overline{C}(b\cdot k_4)$

For to prove this suppose the contrary $(z\cdot x)C(b\cdot k_4)$. We also have $z\cdot x\leq y\cdot x$, $b\cdot k_4\leq b$, so by axiom $(C2)$ we get $(y\cdot x)Cb$ - a contradiction with (16).

From (25), (26) and (27) by axiom $(MC2)$ we obtain $x\widehat{C}k_5$ - a contradiction with (20). Consequently $(\#3)$ is true.

{\bf Step 2.2: the construction of $\Delta$.} Applying the Filter extension Lemma to $(\#3)$ we obtain a prime filter $\Delta$ (and this is just the required $\Delta$) such that:

\begin{enumerate}
\item  $F(x\gg\Gamma)=\{x\in D:(\exists y \in \Gamma)(x\gg y\}\subseteq\Delta$,

\item $F(x\overline{\widehat{C}}\overline{\Gamma})=\{x\in D:(\exists y\in \overline{\Gamma})(x\overline{\widehat{C}}y)\}\subseteq \Delta$,

\item $b\in \Delta$,

\item  $I(x\overline{C}\Gamma)=\{x\in D: (\exists y\in \Gamma)(x\overline{C}y)\}\cap \Delta=\varnothing$,

\item   $I(x\ll \overline{\Gamma})=\{x\in D:(\exists y\in \overline{\Gamma})( x\ll y)\}\cap \Delta=\varnothing$.

\end{enumerate}

\noindent {\bf Step 3: proof of $\Gamma R^{c} \Delta$.} We will verify the four cases  of the definition of $R^{c}$.

\begin{itemize}
\item {\bf Case 1: $ y\in \Gamma$ and $x\in \Delta$.} We have to show $yCx$. Suppose $y\overline{C}x$. Then $x\overline{C}y$ and by $y\in \Gamma$ we get $x\in I(x\overline{C}\Gamma)$. Then by 4. $x\not\in \Delta$ - a contradiction, hence $yCx$.

\item {\bf Case 2: $y\in \Gamma$ and $x\not\in \Delta$.} Suppose $y\ll x$. Then $x\gg y$ and $y\in \Gamma$ implies $x\in F(x\gg \Gamma)$. By (1) $x\in \Delta$ - a contradiction, hence $y\not\ll x$.

\item {\bf Case 3: $y\not\in \Gamma$ and $x\in \Delta$.} Suppose $x\ll y$. Then $x\in I(x\ll\overline{\Gamma})$ and by 5. $x\not\in \Delta$ - a contradiction. Hence $x\not\ll y$.

\item {\bf Case 4: $y\not\in \Gamma$ and $x\not\in \Delta$.} Suppose $y\overline{\widehat{C}}x$. Then $x\overline{\widehat{C}}y$ and by 2. we obtain $x\in \Delta$ - a contradiction. Hence $y\widehat{C}x$.

\end{itemize}

Thus we have constructed prime filters $\Gamma$ and $\Delta$ such that: $a\in \Gamma$, $b\in \Delta$ (item 3 from Step 2.2) and $\Gamma R^{c}\Delta$ (Step 3).

%%%%%%%%%%%%%%%%%%%%%%%%%%%%%%%%%%%%%%%%%%%%%%%%%%%%%%%%%%%%%%%%%%%%%%%%%%%%%%%%%%%%%%%%%%%
\bigskip
%%%%%%%%%%%%%%%%%%%%%%%%%%%%%%%%%%%%%%%%%%%%%%%%%%%%%%%%%%%%%%%%%%%%%%%%%%%%%%%%%%%%%%%%%%%%%%%%%%%%%%%%%%%%%%5
\noindent {\bf Proof of (ii).} ($\Leftarrow$) If $a\in \Gamma$ and $b\not\in \Delta$ then by the definition of $R^{c}$ we obtain  $a\not\ll b$.

$(\Rightarrow)$ Suppose  $a\not\ll b$. The proof, as in (i), will go on several  steps.

\noindent {\bf  Step 1: construction of $\Gamma$.} Consider the ideal  $I(x\ll b)=\{x\in D:x\ll b\}$ (Lemma \ref{FiltersIdeals1}).
%%%%%%%%%%%%%%%%%%%%%%%%%%%%%%%%%%%%%%%%%%%%%%%%%%%%

Since $a\not\ll b$, $a\not\in \{x\in D: x\ll b\} $. Then $[a)\cap \{x\in D: x\ll b\}=\varnothing$ and $[a)$ is a filter (see FACTS \ref{facts}). By the Strong filter-extension lemma (Lemma \ref{strong extension lemma}) there exists a prime filter $\Gamma$ such that  $[a)\subseteq \Gamma$ and $(\forall x\in \Gamma)(x\not\in \{x\in D: x\ll b\})$ and    $(\forall x\not\in \Gamma)(\exists y\in \Gamma)(x\cdot y\in \{x\in D: x\ll b\})$. From here we conclude that $\Gamma$ satisfies the following  properties:

\smallskip
($\#0$) $a\in \Gamma$,

\smallskip

($\#1$) If $x\in \Gamma$, then $x\not \ll b$, and

\smallskip

($\#2$) If $x\not\in \Gamma$, then there exists $y\in \Gamma$ such that $(x\cdot y)\ll b$.

\noindent {\bf  Step 2: construction of $\Delta$.} This will be done in two sub-steps.

{\bf  Step 2.1} Consider the filters and ideals definable by $\Gamma$ as in Lemma \ref{FiltersIdeals2}

 \noindent $I(x\overline{C}\Gamma)=\{x\in D: (\exists y\in \Gamma)(x\overline{C}y)\}$, $F(x\overline{\widehat{C}}\overline{\Gamma})=\{x\in D:(\exists y\in \overline{\Gamma})(x\overline{\widehat{C}}y)\}$, $I(x\ll \overline{\Gamma})=\{x\in D:(\exists y\in \overline{\Gamma})( x\ll y)\}$, and
$F(x\gg\Gamma)=\{x\in D:(\exists y \in \Gamma)(x\gg y\}$.
In order to apply the Filter-extension Lemma (Lemma \ref{feL})  we will prove the following condition:

$(\#3)$   $F(x\gg\Gamma)\oplus F(x\overline{\widehat{C}}\overline{\Gamma})\cap I(x\ll \overline{\Gamma})\oplus I(x\overline{C}\Gamma)\oplus(b]=\varnothing$

Suppose that $(\#3)$ is not true. Consequently $\exists t$ such that\\
(1) $t=k_1\cdot k_2=k_4+k_5+k_6$ for some $k_1,k_2,k_4,k_5,k_6\in D$ and\\
(2) $\exists x_1\in\Gamma$ such that $x_1\ll k_1$,\\
(3) $\exists x_2\in\overline{\Gamma}$ such that $k_2\overline{\widehat{C}}x_2$,\\
(4) $\exists x_3\in\overline{\Gamma}$ such that $k_4\ll x_3$,\\
(5) $\exists x_4\in\Gamma$ such that $k_5\overline{C}x_4$,\\
(6) $k_6\leq b$.

Let $z=x_2+x_3$. Then by (3) and (4) we obtain  $z\in\overline{\Gamma}$. By axiom $(\widehat{C}2)$ we get\\
(7) $k_2\overline{\widehat{C}}z$.

By (4) and  axiom $(\ll4)$ we get\\
(8) $k_4\ll z$.

By $z\not\in\Gamma$ and $(\#2)$ we have\\
(9) $\exists y\in\Gamma$ such that $(z\cdot y)\ll b$.

Let $x=x_1\cdot x_4\cdot y\cdot a$. Then by $(\# 0)$, (2), (5) and (9) we get $x\in\Gamma$. By axiom $(\ll4)$ we get\\
(10) $x\ll k_1$.

By (5), $x\leq x_{4}$ and axiom $(C2)$ we get\\
(11) $k_5\overline{C}x$.

 From $x\in\Gamma$  by $(\#1)$ we obtain\\
(12) $x\not\ll b$.

From (10) by axiom $(\ll4)$ we get\\
(13) $x\ll (b+k_1)$

From (7) by axiom $(\widehat{C}2)$ we obtain\\
(14) $z\overline{\widehat{C}}(b+k_2)$.

From (9) by axiom $(\ll4)$ we get\\
(15) $(z\cdot y)\ll (b+k_2)$.

From (14) and (15) by axiom $(M\ll1)$ we obtain $y\ll (b+k_2)$. We also have $x\leq y$ and by axiom $(\ll4)$ we get\\
(16) $x\ll (b+k_2)$.

From (13) and (16) by axiom $(\ll6)$ we get $x\ll (b+k_1)\cdot(b+k_2)$. We have $(b+k_1)\cdot(b+k_2)=b+k_1\cdot k_2=b+k_4+k_5+k_6=b+k_4+k_5$ (since $k_6\leq b$ from (6)). Thus:\\
(17) $x\ll (b+k_4+k_5)$.

Suppose  (in order to obtain a contradiction) that $x\ll (b+k_4)$. From (9) and $x\cdot z\leq z\cdot y$ (which follows from the definitions of $x$ and $z$) by axiom $(\ll4)$ we obtain $(x\cdot z)\ll b$. Using this fact, (8), $x\ll (b+k_4)$ and axiom $(\ll7)$ we get $x\ll b$ - a contradiction with (12). Consequently\\
(18) $x\not\ll (b+k_4)$.

From (11) and (17) by axiom $(M\ll2)$ we obtain $x\ll (b+k_4)$ - a contradiction with (18). Consequently $(\#3)$ is true.

\smallskip

{\bf Step 2.2: the construction of $\Delta$.} Applying the Filter-extension Lemma to $(\#3)$ we obtain a prime filter $\Delta$ (and this is just the required $\Delta$) such that:

\begin{enumerate}
\item  $F(x\gg\Gamma)=\{x\in D:(\exists y \in \Gamma)(x\gg y\}\subseteq\Delta$,

\item $F(x\overline{\widehat{C}}\overline{\Gamma})=\{x\in D:(\exists y\in \overline{\Gamma})(x\overline{\widehat{C}}y)\}\subseteq \Delta$,

\item $b\not\in \Delta$,

\item  $I(x\overline{C}\Gamma)=\{x\in D: (\exists y\in \Gamma)(x\overline{C}y)\}\cap \Delta=\varnothing$,

\item   $I(x\ll \overline{\Gamma})=\{x\in D:(\exists y\in \overline{\Gamma})( x\ll y)\}\cap \Delta=\varnothing$.

\end{enumerate}
\noindent {\bf Step 3: proof of $\Gamma R^{c} \Delta$.} The proof is the same as in the corresponding step in (i).

\smallskip

To conclude: we have constructed prime filters $\Gamma, \Delta$ such that $\Gamma R^{c}\Delta$, $a\in \Gamma$ and $b\not\in \Delta$, which finishes the proof of the lemma.$\square$

\end{proof}

\begin{lemma}\label{dual C}
(i) $a\widehat{C}b$ iff $(\exists \Gamma,\Delta \in PI(D))(a\in \Gamma$ and $b\in \Delta$ and $\Gamma\widehat{R}^{c}\Delta)$.

(ii) $a\widehat{C}b$ iff $(\exists \Gamma,\Delta \in PF(D))(a\not\in \Gamma$ and $b\not\in \Delta$ and $\Gamma R^{c}\Delta)$.

(iii) $a\not\gg b$ iff $(\exists \Gamma,\Delta \in PI(D))(a\in \Gamma$ and $b\not\in \Delta$ and $\Gamma\widehat{R}^{c}\Delta)$.

(iv) $a\not\gg b$ iff $(\exists \Gamma,\Delta \in PF(D))(a\not\in \Gamma$ and $b\in \Delta$ and $\Gamma R^{c}\Delta)$.
\end{lemma}

\begin{proof} (i) by duality from Lemma \ref{C}. Note that in this case Strong ideal-extension Lemma is used. The proof can follow in a "dual way" the steps of the proof of Lemma \ref{C} (i).

(ii) is a corollary from (i).

(iii) by duality from Lemma \ref{C} (ii) with the same remark as above.

(iv) is a corollary from (iii).$\square$

\end{proof}

\begin{lemma}\label{canonical embedding}  Let $(W^{c}, R^{c})$ be the canonical structure of $\underline{D}=(D,C, \widehat{C}, \ll)$ and  $h(a)=\{U\in PF(D):a\in U\} $  be the Stone mapping from $D$ into the distributive lattice of all subsets of $W^{c}$. Then $h$ is an embedding of $\underline{D}$ into the EDC-lattice over $(W^{c}, R^{c})$.

\end{lemma}

\begin{proof} It is a well known fact that $h$ is an embedding of  distributive lattice into the distributive lattice of all subsets of the set of prime filters $PF(D)$ (see, \cite{Stone,bd}). The only thing which have to be done is to show the following equivalences for all $a,b\in D$:

(i) $aCb$ iff $h(a)C_{R^{c}} h(b)$,

(ii) $a\widehat{C}b$ iff $h(a)\widehat{C}_{R^{c}}h(b)$

(iii) $a\ll b$ iff $h(a) \ll_{R^{c}}h(b)$.

Note that these equivalences are another equivalent reformulation of Lemma \ref{C} (i) and (ii) and Lemma \ref{dual C} (ii) and (iv). $\square$

\end{proof}

\begin{theorem}\label{relational representation} {\bf Relational representation Theorem of EDC-latices.}  Let $\underline{D}=(D,C,\widehat{C},\ll)$ be an EDC-lattice. Then there is a relational system $\underline{W}=(W,R)$ with reflexive and symmetric $R$ and an embedding $h$ into the EDC-lattice of all subsets of $W$.
\end{theorem}

\begin{proof} The theorem is a corollary of Lemma \ref{canonical embedding}.$\square$

\end{proof}

\begin{corollary}\label{embedding in contact algebras} Every EDC-lattice can be isomorphically embedded into a contact algebra.

\end{corollary}

\begin{proof} Since the lattice of all subsets  of a given set is a Boolean algebra, then this is a corollary of Theorem \ref{relational representation}. $\square$

\end{proof}

The following theorem states that the axiom system of EDC-lattice can be considered as an axiomatization of the universal fragment of contact algebras in the language of EDC-lattices.

\begin{theorem} \label{Universal fragment}  Let $A$ be an universal first-order formula in the language of EDC-lattices. Then $A$ is a consequence from the axioms of EDC-lattice iff $A$ is true in all contact algebras.

\end{theorem}

\begin{proof} The proof is a consequence from  Corollary \ref{embedding in contact algebras} and   the fact that all axioms of EDC-lattice are universal first-order conditions and that $A$ is also an universal first-order condition.$\square$
\end{proof}
%%%%%%%%%%%%%%%%%%%%%%%%%%%%%%%%%%%%%%%%%%%%%%%%%%%%%%%
\section{Relations to other mereotopologies}
%%%%%%%%%%%%%%%%%%%%%%%%%%%%%%%%%%%%%%%%%%%%%%%%%%%%%%%%%%%%%%%%%%%%%%%%%%%%%%%%%%%%%%%%%%55
In this section we will compare EDC-lattices with other two mereotopologies: the \emph{relational mereotopology} and \emph{RCC-8}.

\subsection{Relational mereotopology}

Relational mereotopology is based on \emph{mereotopological structures} introduced in \cite{NV}. These are relational structures in the form $(W, \leq, O, \widehat{O}, \ll, C, \widehat{C} )$ axiomatizing the basic mereological relations part-of $\leq$, overlap $O$ and dual overlap (underlap) $\widehat{O}$, and the basic mereotopological relations non-tangential part-of $\ll$, contact $C$ and dual contact $\widehat{C}$. These relations satisfy the following list of universal first-order axioms:

\(\begin{array}{llll}
        (\leq 0)    &   a\leq b$ and $b\leq a \rightarrow a=b   & (\leq 1)    &   a\leq a,\\
        (\leq 2)    &   a\leq b$ and $b\leq c \rightarrow a\leq c\\\\
        (O1)    &   aOb \rightarrow bOa                     &(\widehat{O}1)      &   a\widehat{O}b \rightarrow  b\widehat{O}a\\
        (O2)    &   aOb \rightarrow aOa                     &(\widehat{O}2)      &   a\widehat{O}b \rightarrow a\widehat{O}a\\
        (\overline{O}\leq) & a\overline{O}a\rightarrow a\leq b
        &(\overline{\widehat{O}}\leq)&
        b\overline{\widehat{O}}b\rightarrow a\leq b \\
        (O\leq) &   aOb$ and $b\leq c \rightarrow  aOc      &(\widehat{O}\leq)   &   c\leq a$ and $a\widehat{O}b \rightarrow c\widehat{O}b\\\\
        (O\widehat{O})  &   aOa$ or $a\widehat{O}a  &   (\leq O\widehat{O}) &   c\overline{O}a$ and $c\overline{\widehat{O}}b \rightarrow a\leq b\\\\
        (C)    &   aCb\rightarrow  bCa                  &(\widehat{C})               &   a\widehat{C}b \rightarrow b\widehat{C}a\\
        (CO1)   &   aOb \rightarrow aCb                 &(\widehat{C}\widehat{O}1)   &   a\widehat{O}b \rightarrow a\widehat{C}b\\
        (CO2)   &   aCb\rightarrow aOa                  &(\widehat{C}\widehat{O}2)   &   a\widehat{C}b \rightarrow a\widehat{O}a\\
        (C\leq) &   aCb$ and $b\leq c \rightarrow aCc   &(\widehat{C}\leq)           &   a\widehat{C}b $ and $c\leq b \rightarrow a\widehat{C}c\\\\
        (\ll\leq 1)         &   a\ll b \rightarrow a\leq b\\
        (\ll\leq 2)         &   a\leq b$ and $b\ll c \rightarrow a\ll c &(\ll\leq 3)        &   a\ll b$ and $b\leq c \rightarrow a\ll c\\\\
        (\ll O)             &   a\overline{O}a \rightarrow  a\ll b      &(\ll \widehat{O})  &   b\overline{\widehat{O}}b \rightarrow a\ll b\\
        (\ll CO)            &   aCb$ and $b\ll c \rightarrow aOc & (\ll \widehat{C}\widehat{O})
                                                                                &   c\ll a$  and $ a\widehat{C}b \rightarrow c\widehat{O}b\\
        (\ll C\widehat{O})  &   c\overline{C}a$ and $ c\overline{\widehat{O}}b \rightarrow a\ll b   &(\ll \widehat{C} O)
                                                                        &   c\overline{O}a$ and $c\overline{\widehat{C}}b \rightarrow a\ll b.
    \end{array}\)\\\\

Note that all axioms of mereotopological structures  are universal first-order conditions which are true in contact algebras under the standard definitions of the three basic mereological relations.

It is proved in \cite{NV} that each mereotopological structure is embeddable into a contact algebra (Theorem 26).

The following theorem relates EDC-lattices to mereotopological structures.

\begin{theorem} Every EDC-lattice is a mereotopological structure under the standard definitions of the basic mereological relations.

\end{theorem}

\begin{proof} Since all axioms of mereotopological structures are universal first-order sentences true in all contact algebras, then the statement follows from Theorem \ref{Universal fragment}.

\end{proof}

%%%%%%%%%%%%%%%%%%%%%%%%%%%%%%%%%%%%%%%%
\subsection{RCC-8 spatial relations}
%%%%%%%%%%%%%%%%%%%%%%%%%%%%%%%%%%%%%%%%%%%%%%%%%%%%

One of the most popular systems of topological relations in the
community of QSRR is RCC-8. The system  RCC-8 was introduced for the first time in \cite{EF}. It consists of 8
relations between non-empty regular closed subsets of arbitrary
topological space. Having in mind the topological representation
of contact algebras, it was given  in \cite{V} an equivalent
definition of RCC-8 in the language of contact algebras:

\begin{definition}\label{RCC-8Def} { The system \bf RCC-8.}
\begin{description}

\item[$\bullet$ disconnected --  DC$(a,b)$:]  $a\overline{C}b$,

\item[$\bullet$ external contact -- EC$(a,b)$:] $aCb$ and
$a\overline{O}b$,

\item[$\bullet$ partial overlap -- PO$(a,b)$:] $aOb$ and
$a\not\leq b$ and ${b\not \leq a}$,

\item[$\bullet$ tangential proper part -- TPP$(a,b)$:] $a\leq b$
and $a\not\ll b$ and ${b\not\leq a}$,

\item[$\bullet$  tangential proper part$^{-1}$ --
TPP$^{-1}(a,b)$:] $b\leq a$ and $b\not\ll a$ and ${a\not\leq b}$,

\item[$\bullet$ nontangential proper part NTPP$(a,b)$:] $a\ll b$
and ${a\not= b}$,

\item[$\bullet$ nontangential proper part$^{-1}$ --
NTPP$^{-1}(a,b)$:] $b\ll a$ and ${a\not=b}$,

\item[$\bullet$ equal -- EQ$(a,b)$:] $a=b$.

\end{description}
\end{definition}

\begin{figure}[h]
\centering
\begin{center}
 \includegraphics{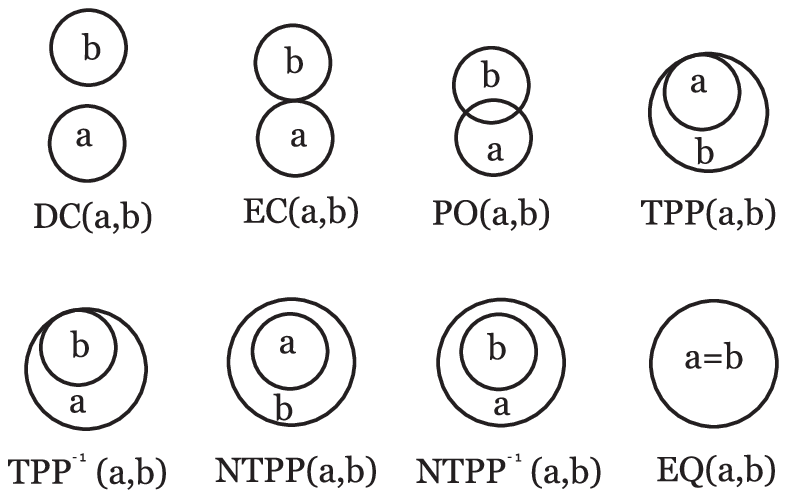}
\end{center}

\begin{center}{\bf RCC-8 relations}
\end{center}

 \end{figure}

Looking at this definition it can be easily seen that the RCC-8 relations are expressible in the language of EDC-lattices. Let us note that RCC-8 relations are not expressible in the language of distributive contact algebras from \cite{dmvw2}.

%%%%%%%%%%%%%%%%%%%%%%%%%%%%%%%%%%%%%%%%%%%%%%%%%%%%%%%%%%%%5
\section{Additional axioms}\label{additional axioms}
%%%%%%%%%%%%%%%%%%%%%%%%%%%%%%%%%%%%%%%%%%%%%%%%%%%%%%%%%%%%%%%%%
In this Section we will formulate several additional axioms for EDC-lattices which are  adaptations for the language of EDC-lattices of some known axioms considered in the context of contact algebras. First we will formulate some new lattice axioms for EDC-lattices - the so called extensionality axioms for the definable predicates of overlap - $aOb \leftrightarrow_{def} a\cdot b\not=0$ and underlap - $a\widehat{O}b \leftrightarrow_{def} a+b\not= 1$.

\noindent (Ext O)  $a\not\leq b\rightarrow (\exists c)(a\cdot c\not=0$ and $b\cdot c=0)$ - \emph{extensionality of overlap},

\smallskip

\noindent (Ext $\widehat{O}$)  $a\not\leq b\rightarrow (\exists c)(a+c=1$ and $b+c\not=1)$ - \emph{extensionality of underlap}.

\smallskip

We say that a lattice is \emph{O-extensional} if it satisfies (Ext O) and \emph{U-extensional} if it satisfies (Ext $\widehat{O}$). Note that the  conditions (Ext O) and (Ext $\widehat{O}$) are true in Boolean algebras but not always are true in distributive lattices (see \cite{dmvw2} for some examples, references and additional information about these axioms).

\smallskip

We will study also the following extensionality axioms.

\smallskip

\noindent (Ext C) $a\not=1\rightarrow(\exists b\not=0)(a\overline{C}b)$ -\emph{C-extensionality},

\smallskip

\noindent (Ext $\widehat{C}$) $a\not=0\rightarrow(\exists b\not=1)(a\overline{\widehat{C}}b)$ -\emph{$\widehat{C}$-extensionality}.

\smallskip

In contact algebras these two axioms are equivalent. It is proved in \cite{dmvw2} that (Ext $\widehat{O}$) implies that (Ext C) is equivalent to
the following extensionality principle considered by Whitehead \cite{W}

\smallskip

\noindent (EXT C) $a\not\leq b\rightarrow (\exists c)(aCc$ and $b\overline{C}c)$.

\smallskip

Just in a dual way one can show that (Ext O) implies that (Ext $\widehat{C}$) is equivalent to the following condition

\smallskip

\noindent (EXT $\widehat{C}$) $a\not\leq b\rightarrow (\exists c)(b\widehat{C}c$ and $a\overline{\widehat{C}}c)$.

\smallskip

Let us note that (EXT C) and (EXT $\widehat{C}$ ) are equivalent in contact algebras.

\smallskip

\noindent (Con C) $a\not=0, b\not=0$ and $a+b=1\rightarrow aCb$ - \emph{C-connectedness axiom } and

\smallskip

\noindent (Con $\widehat{C}$)  $a\not=1, b\not=1$ and $a\cdot b=0\rightarrow a\widehat{C}b$ - \emph{$\widehat{C}$-connectedness axiom }.

\smallskip

In contact algebras these axioms are equivalent and guarantee topological representation in connected topological spaces.

\smallskip

\noindent (Nor 1)  $a\overline{C}b\rightarrow (\exists c,d)(c+d=1, a\overline{C}c$ and $b\overline{C}d)$,

\smallskip

\noindent (Nor 2) $a\overline{\widehat{C}}b\rightarrow (\exists c,d)(c\cdot b=0, a\overline{\widehat{C}}c$ and $b\overline{\widehat{C}}d)$,

\smallskip

\noindent (Nor 3) $a\ll b\rightarrow (\exists c)(a\ll c\ll b)$.

\smallskip

Let us note that the above three axioms are equivalent in contact algebras and are known by different names.  For instance (Nor 1) comes from the proximity theory \cite{Thron} as \emph{Efremovich axiom}, (Nor 3) sometimes is called \emph{interpolation axiom}. We adopt the name \emph{normality axioms} for (Nor 1), (Nor 2) and (Nor 3) because in topological representations they imply some normality conditions in the corresponding topological spaces. It is proved in \cite{DuV} that (Nor 1) is true in the relational models $(W, R)$ (see Section \ref{relmodels}) if and only if the relation $R$ is transitive and that (Nor 1) implies representation theorem in transitive models. In the next lemma we shall prove similar result using all normality axioms.

\begin{lemma} \label{transitivity lemma} {\bf Transitivity lemma.} Let $\underline{D}=(D,C,\widehat{C},\ll)$ be a EDC-lattice satisfying the axioms (Nor1), (Nor 2) and (Nor 3) and let $(W^{c},R^{c})$ be the canonical structure of $\underline{D}$ (see Section \ref{relrepresentation}) Then:

(i) $R^{c}$ is a transitive relation.

(ii) $\underline{D}$ is representable in EDC-lattice over some system $(W,R)$ with an equivalence relation $R$.

\end{lemma}

\begin{proof} (i)  Let $\Gamma,\Delta$ and $\Theta$ be prime filters in $D$ such that

(1) $\Gamma R^{c}\Delta$ and

(2) $\Delta R^{c}\Theta$

and suppose for the sake of contradiction that

(3) $\Gamma\overline{R}^{c}\Theta$. By the definition of $R^{c}$ we have to consider four cases.

\smallskip

\noindent {\bf Case 1:} $\exists a\in \Gamma, b\in \Theta$ such that $a\overline{C}b$.
\begin{quote}Then by (Nor 1) there exists $c,d$ such that $c+d=1$, $a\overline{C}c$ and $b\overline{C}d$. Since $c+d=1$ then either $c\in \Delta$ or $d\in \Delta$. The case $c\in \Delta$ together with $a\in \Gamma$ imply  by (1) $aCc$ - a contradiction. The case $d\in \Delta$ together with $b\in \Theta$ imply by (2) $bCd$ - again a contradiction.
\end{quote}

\noindent {\bf Case 2:} $\exists a\in \Gamma, b\not\in \Theta$ such that $a\ll b$.

\begin{quote} Then by (Nor 3) $\exists c$ such that $a\ll c$ and $c\ll b$. Consider the case $c\not \in\Delta$. Then $a\in \Gamma$ and (1) imply $a\not \ll c$ a contradiction. Consider now $c\in \Delta$. Then $b\not\in \Theta$ imply $c\not \ll b$ - again a contradiction.

\end{quote}

In a similar way one can obtain a contradiction in the remaining two cases:

\noindent {\bf Case 3:} $\exists a\not\in \Gamma, b\in \Theta$ such that $b\ll a$ and

 \noindent {\bf Case 4:} $\exists a\not\in \Gamma, b\not\in \Theta$ such that $b\overline{\widehat{C}} a$.

(ii) The proof follows from (i) analogous to the proof of Theorem \ref{relational representation}.$\square$
\end{proof}

\medskip

Another kind of axioms which will be used in the topological representation theory in PART II are the so called rich axioms.

\smallskip

\noindent (U-rich $\ll $)  $a\ll b\rightarrow (\exists c)(b+c=1$ and $a\overline{C}c)$,

\smallskip

\noindent (U-rich $\widehat{C}$) $a\overline{\widehat{C}}b\rightarrow (\exists c, d)(a+c=1, b+d=1$ and $c\overline{C}d)$.

\smallskip

\noindent (O-rich $\ll$)  $a\ll b\rightarrow (\exists c)(a\cdot c=0$ and $c\overline{\widehat{C}}b)$,

\smallskip

\noindent (O-rich C)   $a\overline{C}b\rightarrow (\exists c,d)(a\cdot c=0, b\cdot d=0$ and $c\overline{\widehat{C}}d)$.

\smallskip

Let us note that U-rich axioms will be used always with the U-extensionality axiom and that O-rich axioms will be used always with O-extensionality axiom.

The following lemma is obvious.

 \begin{lemma}\label{rich} The axioms (U-rich $\ll $), (U-rich $\widehat{C}$), (O-rich $\ll$) and (O-rich C) are true in all contact algebras.

 \end{lemma}

%%%%%%%%%%%%%%%%%%%%%%%%%%%%%%%%%%%%%%%%%%%%%%%%%%%%%%%%%%%%%%%%%%%%%%
\subsection{Some good embedding properties}\label{embedding properties}
%%%%%%%%%%%%%%%%%%%%%%%%%%%%%%%%%%%%%%%%%%%%%%%%%%%%%%%%%%%%%%%%%%%%%%

Let $(D_{1}, C_{1}, \widehat{C}_{1}, \ll_{1})$ and $(D_{2}, C_{2}, \widehat{C}_{2}, \ll_{2})$ be two EDC-lattices. We will write $D_1 \preceq D_2$ if $D_{1}$ is a
substructure of $D_{2}$, i.e., $D_1$ is a sublattice of $D_2$, and
the relations $C_{1}, \widehat{C}_{1}, \ll_{1}$ are restrictions of the relations $C_{2}, \widehat{C}_{2}, \ll_{2}$ on $D_{1}$.  Since we want to prove embedding theorems, it is valuable to know under what
conditions  we have
equivalences of the form:
\begin{center}
$D_{1}$ satisfies some additional axiom iff  $D_{2}$ satisfies the same axiom.
\end{center}

\begin{definition}\label{dense and dual dense} {\bf Dense and dual dense sublattice.} Let $D_{1}$ be a  distributive sublattice of $D_{2}$. $D_{1}$ is
called a \emph{ dense} sublattice of $D_{2}$ if the following
condition is satisfied:

\smallskip

\noindent(Dense) \quad $(\forall a_{2}\in D_{2})(a_{2}\not=0
\Rightarrow(\exists a_{1}\in D_{1})( a_{1}\leq a_{2}$ and
$a_{1}\not=0))$.

\smallskip

If $h$ is an embedding of the lattice $D_{1}$ into the lattice
$D_{2}$ then we say that $h$ is a \emph{ dense} embedding if the
sublattice $h(D_{1})$ is a  dense sublattice of $D_{2}$.

\smallskip

Dually, $D_{1}$ is
called a \emph{dual dense} sublattice of $D_{2}$ if the following
condition is satisfied:

\smallskip

\noindent(Dual dense) \quad $(\forall a_{2}\in D_{2})(a_{2}\not=1
\Rightarrow(\exists a_{1}\in D_{1})( a_{2}\leq a_{1}$ and
$a_{1}\not=1))$.

\smallskip

If $h$ is an embedding of the lattice $D_{1}$ into the lattice
$D_{2}$ then we say that $h$ is a \emph{Dual dense} embedding if the
sublattice $h(D_{1})$ is a dually dense sublattice of $D_{2}$.

\end{definition}

Note that in Boolean
algebras, dense and dually dense conditions are equivalent;
in distributive lattices this equivalence does not hold (see \cite{dmvw2} for some known characterizations of density and dual density in distributive lattices).

\smallskip

For the case of contact algebras \cite{V} and distributive contact lattices \cite{dmvw2} we introduced the notion of \emph{C-separability} as follows. Let $D_{1} \preceq D_{2}$; we say that $D_{1}$ is a
\emph{$C$-separable sublattice of $D_{2}$} if
the following condition is satisfied:

\smallskip

\noindent(C-separable) $(\forall a_{2},b_{2}\in
D_{2})(a_{2}\overline{C}b_{2} \Rightarrow (\exists a_{1},b_{1}\in
D_{1})(a_{2}\leq a_{1}, \/ b_{2}\leq b_{1},
\/a_{1}\overline{C}b_{1}))$.

For the case of EDC-lattices we modified this notion adding two additional clauses corresponding to  the relations $\widehat{C}$ and $\ll$ just having in mind the definitions of these relations in contact algebras. Namely

\begin{definition}\label{C-separability1} {\bf C-separability.} Let $D_{1} \preceq D_{2}$; we say that $D_{1}$ is a
\emph{$C$-separable EDC-sublattice of $D_{2}$} if
the following conditions are satisfied:

\smallskip

\noindent(C-separability for C) -

$(\forall a_{2},b_{2}\in
D_{2})(a_{2}\overline{C}b_{2} \Rightarrow (\exists a_{1},b_{1}\in
D_{1})(a_{2}\leq a_{1}, \/ b_{2}\leq b_{1},
\/a_{1}\overline{C}b_{1}))$.

\smallskip

\noindent(C-separability for $\widehat{C}$) -

$(\forall a_{2},b_{2}\in
D_{2})(a_{2}\overline{\widehat{C}}b_{2} \Rightarrow (\exists a_{1},b_{1}\in
D_{1})(a_{2}+ a_{1}=1, \/ b_{2}+b_{1}=1,
\/a_{1}\overline{C}b_{1}))$.

\smallskip

\noindent(C-separability for $\ll$) -

$(\forall a_{2},b_{2}\in
D_{2})(a_{2}\ll b_{2} \Rightarrow (\exists a_{1},b_{1}\in
D_{1})(a_{2}\leq a_{1}, \/ b_{2}+ b_{1}=1,
\/a_{1}\overline{C}b_{1}))$.

\smallskip

If $h$ is an embedding of the lattice $D_{1}$ into the lattice
$D_{2}$ then we say that $h$ is a \emph{C-separable embedding} if the
sublattice $h(D_{1})$ is a C-separable sublattice of $D_{2}$.

\end{definition}

The notion of a $C$-separable embedding $h$ is defined similarly. The following lemma
is analogous  to a similar result from \cite{V} (Theorem 2.2.2) and from  \cite{dmvw2} (Lemma 5).

\begin{lemma}\label{dense}
Let $D_{1}, D_{2}$ be  EDC-lattices  and $D_{1}$ be a
$C$-separable EDC-sublattice of $D_{2}$. Then:

(i) If $D_{1}$ is a dually dense EDC-sublattice of $D_{2}$, then $D_{1}$
satisfies the axiom (Ext C) iff  $D_{2}$ satisfies the axiom
(Ext C),

(ii) $D_{1}$ satisfies the axiom (Con C) iff  $D_{2}$ satisfies the
axiom (Con C),

(iii) $D_{1}$ satisfies the axiom (Nor 1) iff  $D_{2}$ satisfies the
axiom (Nor 1),

(iv) $D_{1}$ satisfies the axiom (U-rich $\ll$) iff  $D_{2}$ satisfies the
axiom (U-rich $\ll$),

(v) $D_{1}$ satisfies the axiom (U-rich $\widehat{C}$) iff  $D_{2}$ satisfies the
axiom (U-rich $\widehat{C}$).
\end{lemma}

\begin{proof} Conditions (i), (ii) and (iii) have the same proof as in Theorem 2.2.2 from  \cite{V}.

(iv) ($\Rightarrow$) Suppose that $D_{1}$ satisfies the axiom (U-rich $\ll$), $a_{2},b_{2}\in D_{2}$ and let $a_{2}\ll b_{2}$. Then by (C-separability for $\ll$) we obtain: $(\exists a_{1},b_{1}\in D_{1})(a_{2}\leq a_{1}, b_{2}+b_{1}=1, a_{1}\overline{C}b_{1})$. Since $D_{1}$ is a sublattice of $D_{2}$ then $a_{1},b_{1}\in D_{2}$. From $a_{2}\leq a_{1}$ and $a_{1}\overline{C}b_{1}$ we get $a_{2}\overline{C}b_{1}$. Thus we have just proved: $(a_{2}\ll b_{2}\rightarrow (\exists b_{1}\in D_{2})(b_{2}+b_{1}=1$ and $ a_{2}\overline{C}b_{1})$ which shows that $D_{2}$ satisfies (U-rich $\ll$).

($\Leftarrow$) Suppose that $D_{2}$ satisfies the axiom (U-rich $\ll$),  $a_{1},b_{1}\in D_{1}$ (hence $a_{1},b_{1}\in D_{2}$) and let $a_{1}\ll b_{1}$. Then by (U-rich $\ll$) for $D_{2}$  we get: $(\exists c_{2}\in D_{2})(b_{1}+c_{2}=1, a_{1}\overline{C}c_{2})$. Since $a_{1},c_{2}\in D_{2}$ and $a_{1}\overline{C}c_{2}$, then by (C-separability for C) we get: $(\exists a_{1}',b_{1}'\in D_{1})(a_{1}\leq a_{1}', c_{2}\leq b_{1}', a_{1}'\overline{C}b_{1}')$. Combining the above results we get: $1=b_{1}+c_{2}\leq b_{1}+b_{1}'$ and $a_{1}\overline{C}b_{1}'$. We have just proved the following: $a_{1}\ll b_{1}\rightarrow (\exists b_{1}'\in D_{1})(b_{1}+b_{1}'=1,a_{1}\overline{C}b_{1}')$ which shows that $D_{1}$ satisfies (U-rich $\ll$).

(v) The proof is similar to that of (iv). $\square$.
\end{proof}

The notion of $\widehat{C}$-separable sublattice can be defined in a dual way as follows:

\begin{definition}\label{dualCseparable} Suppose that $D_{1} \preceq D_{2}$; we say that $D_{1}$ is a
\emph{$\widehat{C}$-separable EDC-sublattice of $D_{2}$} if
the following condition is satisfied:

\smallskip

\noindent($\widehat{C}$-separability for C) -

$(\forall a_{2},b_{2}\in
D_{2})(a_{2}Cb_{2} \Rightarrow (\exists a_{1},b_{1}\in
D_{1})(a_{1}+a_{2}=1, \/ b_{1}+ b_{2}=1,
\/a_{1}\overline{\widehat{C}}b_{1}))$,

\smallskip

\noindent($\widehat{C}$-separability for $\widehat{C}$) -

$(\forall a_{2},b_{2}\in
D_{2})(a_{2}\widehat{C}b_{2} \Rightarrow (\exists a_{1},b_{1}\in
D_{1})(a_{1}\leq
a_{2}, \/ b_{1}\leq b_{2},
\/a_{1}\overline{\widehat{C}}b_{1}))$,

\smallskip

\noindent($\widehat{C}$-separability for $\ll$) -

$(\forall a_{2},b_{2}\in
D_{2})(a_{2}\ll b_{2} \Rightarrow (\exists a_{1},b_{1}\in
D_{1})(a_{1}+a_{2}=1, \/ b_{1}\leq b_{2},
\/a_{1}\overline{\widehat{C}}b_{1}))$.

\smallskip

The notion of a $\widehat{C}$-separable embedding $h$ is defined as in definition \ref{C-separability1}.

\end{definition}

The following lemma is dual to Lemma \ref{dense} and  can be proved in a dual way.

\begin{lemma}\label{dual dense}
Let $D_{1}, D_{2}$ be  EDC-lattices and $D_{1}$ be a
$\widehat{C}$-separable EDC-sublattice of $D_{2}$; then:

\smallskip

(i) If $D_{1}$ is a  dense EDC-sublattice of $D_{2}$, then $D_{1}$
satisfies the axiom (Ext $\widehat{C}$) iff  $D_{2}$ satisfies the axiom
(Ext $\widehat{C}$),

(ii) $D_{1}$ satisfies the axiom (Con $\widehat{C}$) iff  $D_{2}$ satisfies the
axiom (Con $\widehat{C}$),

(iii) $D_{1}$ satisfies the axiom (Nor 2) iff  $D_{2}$ satisfies the
axiom (Nor 2).

(iv) $D_{1}$ satisfies the axiom (O-rich $\ll$) iff  $D_{2}$ satisfies the
axiom (O-rich $\ll$).

(v) $D_{1}$ satisfies the axiom (O-rich $\widehat{C}$) iff  $D_{2}$ satisfies the
axiom (O-rich $\widehat{C}$).

\end{lemma}

\begin{corollary}\label{limitation} Let $\underline{D}=(D,C,\widehat{C},\ll)$ be an EDC-lattice and $\underline{B}=(B,C)$ be a contact algebra. Then:

(i) If $h$ is a C-separable embedding of $\underline{D}$ into $\underline{B}$ then $\underline{D}$ must satisfy the axioms (U-rich $\ll$) and (U-rich $\widehat{C}$).

(ii) If $h$ is a $\widehat{C}$-separable embedding of $\underline{D}$ into $\underline{B}$ then $\underline{D}$ must satisfy the axioms (O-rich $\ll$) and (O-rich $\widehat{C}$).

\end{corollary}

\begin{proof} (i)  Note that by Lemma \ref{rich} $\underline{B}$ satisfies the axioms (U-rich $\ll$) and (U-rich $\widehat{C}$). Then by Lemma \ref{dense} (iv) and (v) $\underline{D}$ satisfies the axioms (U-rich $\ll$) and (U-rich $\widehat{C}$).

(ii) Similarly to (i) the proof follows from Lemma \ref{rich} and Lemma \ref{dual dense}. $\square$

\end{proof}

\bigskip

\bigskip

\bigskip

\noindent {\bf PART II: TOPOLOGICAL REPRESENTATIONS OF EXTENDED DISTRIBUTIVE CONTACT LATTICES  }

\bigskip

\bigskip

The aim of this second part of the paper is to investigate several kinds of topological representations of EDC-lattices. We concentrate our attention mainly on topological representations with some  "good properties" in the sense of Section \ref{embedding properties}: dual density and C-separability, and their dual versions - density and $\widehat{C}$-separability.

\section{Topological models of EDC-lattices} We assume some familiarity of the reader with the basic theory of topological spaces:(see \cite{E}). First we recall some notions from topology. By a
topological space  we mean a set $X$ provided with a
family $C(X)$ of subsets, called closed sets, which contains the
empty set $\varnothing$,  the whole set $X$, and is closed with
respect to finite unions and arbitrary intersections. Fixing $C(X)$ we say that
$X$ is endowed with a topology. A subset $a\subseteq X$ is called
{\em open} if it is the complement of a closed set. A family of closed sets $\mathbf{CB}(X)$ is called
a {\em closed basis} of the topology if every closed set can be
represented as an intersection of sets from $\mathbf{CB}(X)$.
 In a similar way the topology of $X$ can be characterized by the family $O(X)$ of open sets: it contains the empty set, $X$ and is closed under finite intersections and arbitrary unions. A family $\mathbf{OB}(X)$  of open sets is called an \emph{open basis} of the topology if every open set can be represented as an union of sets from $\mathbf{OB}(X)$. $X$ is called \emph{semiregular space} if it has a closed  base of regular closed sets or an open base of regular open sets.

We remaind the definitions of two important topological operations on sets - \emph{closure operation} $Cl$, and \emph{interior operation} $Int$. Namely $Cl(a)$ is the intersection of all closed sets of $X$ containing $a$ and $Int(a)$ is the union of all open sets included in $a$. Note that the operations $Cl$ and $Int$ are interdefinable: $Cl(a)=-Int(-a)$ and $Int(a)=-Cl(-a)$.  Using the bases $\mathbf{CB}(X)$  and $\mathbf{OB}(X)$ the definitions of closure and interior operations have the following useful expressions:

\smallskip

$x\in Cl(a)$ iff $(\forall b\in \mathbf{CB}(X))( a\subseteq b\rightarrow x\in b)$,

\smallskip

$x\in Int(a)$ iff $(\exists b\in \mathbf{OB}(X))(b\subseteq a$ and $x\in b)$.

\smallskip

 We say that $a$ is a \emph{regular closed set} if $a=Cl(Int(a))$ and $a$ is a \emph{regular open} set if $a=Int(Cl(a))$. It is a well known fact that the set $RC(X)$ of all regular closed subsets of $X$ is a Boolean algebra with respect to the relations, operations and constants defined as follows: $a\leq b$ iff $a\subseteq b$, $0=\varnothing$, $1=X$, $a+b=a\cup b$, $a\cdot b= Cl(Int(a\cap b)$, $a^{*}=Cl(-a)$ where $-a=X\smallsetminus a$. If we define a contact $C$ by $aCb$ iff $a\cap b\not= \varnothing$ then we obtain the standard topological model of contact algebra.

 Another topological model of contact algebra is by the set $RO(X)$  of regular open subsets of $X$. The relevant definitions are as follows: $a\leq b$ iff $a\subseteq b$, $0=\varnothing$, $1=X$, $a\cdot b= a\cap b$, $a+b= Int(Cl(a\cup b)$, $a^{*}=Int-a$. The contact relation is $aCb$ iff $Cl(a)\cap Cl(b)\not=\varnothing$.

 Note that these two models are isomorphic.

  {\bf Topological model of EDC-lattice by regular-closed sets.} Consider the contact algebra  $RC(X)$ of regular closed subsets of $X$. Let us   remove the operation $a^{*}$ and define the relations $\widehat{C}$ and $\ll$ topologically  according to their definitions in contact algebra as follows:

 \smallskip

$a\widehat{C}b$ iff $ Cl(-a)\cap Cl(-b)\not=\varnothing $ iff (equivalently) $Int(a)\cup Int(b)\not= X$.

\smallskip

 $a\ll b$ iff $a\cap Cl(-b)=\varnothing$ iff (equivalently) $a\subseteq Int(b)$.

 \smallskip

 Obviously the obtained  structure is a model of EDC-lattice. Also any distributive sublattice of $RC(X)$ with the same definitions of the relations $C$, $\widehat{C}$ and $ \ll$ is  a model of EDC-lattice. These models   are considered as  \emph{standard topological models of EDC-lattice by regular closed sets.}

 \smallskip

 {\bf Topological model of EDC-lattice by regular-open sets.} Consider the contact algebra  $RO(X)$ of regular open subsets of $X$. Let us   remove the operation $a^{*}$ from the contact algebra $RO(X)$   and define the relations $\widehat{C}$ and $\ll$ topologically  according to their definitions in the contact algebra as follows:

 \smallskip

  $a\widehat{C}b$ iff $Cl(Int(-a)\cap Cl(Int(-b))\not=\varnothing$ iff (equivalently) $a\cup b\not=X$,

  \smallskip

  $a\ll b$ iff $Cl(a)\cap Cl(Int(-b))=\varnothing$ iff (equivalently) $Cl(a)\subseteq b$.

 Obviously the obtained  structure is another  standard topological model of EDC-lattice and any distributive sublattice of $RO(X)$ with the same relations $C$, $\widehat{C}$ and $ \ll$ is also a model of EDC-lattice.

 The main aim of PART II of the paper is the topological representation theory of EDC-lattices related to the above two standard models. The first simple result is the following representation theorem.

 \begin{theorem}\label{topological representation1}{\bf Topological representation theorem for EDC-lattices.} Let $\underline{D}=(D,C,\widehat{C},\ll)$ be an EDC-lattice. Then:

  (i) There exists a topological space $X$ and an embedding  of $\underline{D}$ into the contact algebra $RC(X)$ of regular closed subsets of $X$.

(ii) There exists a topological space $Y$ and an embedding  of $\underline{D}$ into the contact algebra $RO(Y)$ of regular open subsets of $Y$.
\end{theorem}

\begin{proof} It is shown in \cite{dv} that every contact algebra is isomorphic to a subalgebra of the contact algebra $RC(X)$ of regular closed subsets of some topological space $X$,  and dually, that it is also isomorphic to a subalgebra of the contact algebra $RO(Y)$ of the regular open subsets of some topological space $Y$. Then the proof follows directly from this result and the Corollary \ref{embedding in contact algebras}.$\square$

\end{proof}

The above theorem is not the best one, because it can not be extended straightforwardly to EDC-lattices satisfying some of the additional axioms mentioned in Section \ref{additional axioms}. That is why we will study in the next sections representation theorems based on embeddings satisfying some of the good conditions described in Section \ref{embedding properties}.   Before going on let us remaind some other topological facts, which will be used later on.

 A topological space $X$ is called:

 $\bullet$  {\em normal} if every pair of closed disjoint sets can be
separated by a pair of open sets;

$\bullet$ {\em $\kappa$-normal}  \cite{Shchepin} if
every pair of regular closed disjoint sets can be separated by a
pair of open sets;

$\bullet$  {\em weakly regular}  \cite{dw} if it is
semiregular and for each nonempty open set $a$ there exits a
nonempty open set $b$ such that $Cl(a)\subseteq b$;

$\bullet$ {\em connected} if it can not be represented by a
sum of two disjoint nonempty open sets;

$\bullet$ {\em $T_0$} if for every pair of distinct points there
is an open set containing one of them and not containing the
other; $X$ is called {\em $T_1$} if every one-point set is a
closed set, and $X$ is called {\em Hausdorff } (or {\em $T_2$}) if
each pair of distinct points can be separated by a pair of
disjoint open sets.

$\bullet$ {\em compact} if it satisfies the
following condition: let $\{A_{i}: i\in I\}$ be a non-empty family
of closed sets of $X$ such that for every finite subset
$J\subseteq I$ the intersection $\bigcap\{A_{i}: i\in
J\}\not=\varnothing$, then $\bigcap\{A_{i}: i\in
I\}\not=\varnothing$.

The following lemma relates  topological properties to the properties of the relations $C$, $\widehat{C}$ and $\ll$ and shows the importance of the additional axioms for EDC-lattices.

\begin{lemma}\label{topoequivalents}

(i)  If $X$ is semiregular, then $X$ is weakly regular iff
$RC(X)$ satisfies any of the axioms (Ext C), (Ext $\widehat{C}$).

(ii)  $X$ is $\kappa$-normal iff  $RC(X)$ satisfies any of the axioms
(Nor 1), (Nor 2) and (Nor 3).

(iii)  $X$ is connected iff  $RC(X)$ satisfies any of the  axioms (Con C), (Con $\widehat{C}$).

(iv)  If $X$ is compact and Hausdorff, then  $RC(X)$ satisfies
(Ext C), (Ext $\widehat{C}$) and (Nor 1), (Nor 2) and (Nor 3) .

%(v) If $X$ is Hausdorff and normal, then RO(X) satisfies (Nor).
\end{lemma}

\begin{proof} A variant of the above lemma concerning only axioms (Ext C), (Nor 1) and (Con C) was proved, for instance, in  \cite{dw}. Having in mind the equivalence of some of the mentioned axioms in $RC(X)$, it is obvious that the present formulation is equivalent to the cited result from \cite{dw}.
\end{proof}

%%%%%%%%%%%%%%%%%%%%%%%%%%%%%%%%%%%%%%%%%%%%%%%%%%%%%%%%%%%%%%%%%%%%%%%%%%%%%%%%%%%%%%%%%%%%5
\subsection{Looking for  good topological representations of\\ EDC-lattices}

The following topological theorem proved in \cite{dmvw2} (Theorem 4) gives necessary and sufficient
conditions for a closed base of a topology to be semiregular.

\begin{theorem}\label{semiregularity1} {\bf First characterization theorem for
semiregularity.}\\
Let $X$ be a topological space and let $\mathbf{CB}(X)$ be a closed
basis for $X$. Suppose that "$\cdot$" is a binary operation defined
on the set $\mathbf{CB}(X)$ such that
$(\mathbf{CB}(X),\varnothing,X,\cup,\cdot)$ is a lattice. Then:
\begin{enumerate}
    \item The following conditions are equivalent:
    \begin{enumerate}
        \item $\mathbf{CB}(X)$ is $U$-extensional.
        \item $\mathbf{CB}(X)\subseteq RC(X)$.
        \item For all $a,b\in \mathbf{CB}(X)$, $a\cdot b=Cl(Int(a\cap
        b))$.
        \item $(\mathbf{CB}(X),\varnothing,X,\cup,\cdot)$ is a dually dense sublattice of the Boolean algebra $RC(X)$.
    \end{enumerate}
    \item If any of the (equivalent) conditions (a),(b),(c) or (d)
    of 1.\ is fulfilled then:
    \begin{enumerate}
        \item $(\mathbf{CB}(X),\varnothing,X,\cup,\cdot)$ is a $U$-extensional distributive
        lattice.
        \item $X$ is a semiregular space.
    \end{enumerate}
\end{enumerate}
\end{theorem}

The following is a  corollary of the above theorem.

\begin{corollary}\label{ext:semireg1} \cite{dmvw2}
Let $X$ be a topological space, let $L=(L,0,1,+,\cdot)$ be a lattice
and let $h$ be an embedding of the upper semi-lattice $(L,0,1,+)$
into the lattice $C(X)$ of closed sets of $X$. Suppose that the
set $\mathbf{CB}(X)=\{h(a): a\in L\}$ forms a closed basis for the
topology of $X$. Then:
\begin{enumerate}
    \item The following conditions are equivalent:
    \begin{enumerate}
        \item $L$ is $U$-extensional.
        \item $\mathbf{CB}(X)\subseteq RC(X)$.
        \item For all $a,b\in L$, $h(a\cdot b)=Cl(Int(h(a)\cap
        h(b)))$.
        \item $h$ is a dually dense embedding of $L$ into the Boolean algebra $RC(X)$.
    \end{enumerate}
    \item If any of the (equivalent) conditions (a),(b),(c) or (d)
    of 1.\ is fulfilled then:
    \begin{enumerate}
        \item $L$ is a $U$-extensional distributive lattice.
        \item $X$ is a semiregular space.
    \end{enumerate}
\end{enumerate}
\end{corollary}

A dual version of Theorem \ref{semiregularity1} is the following one.

\begin{theorem}\label{semiregularity} {\bf Second characterization theorem for
semiregularity.}\\
Let $X$ be a topological space and let $\mathbf{OB}(X)$ be an open
basis for $X$. Suppose that $+$ is a binary operation defined
on the set $\mathbf{OB}(X)$ such that
$(\mathbf{OB}(X),\varnothing,X,\cap,+)$ is a lattice. Then:
\begin{enumerate}
    \item The following conditions are equivalent:
    \begin{enumerate}
        \item $\mathbf{OB}(X)$ is $O$-extensional.
        \item $\mathbf{OB}(X)\subseteq RO(X)$.
        \item For all $a,b\in \mathbf{OB}(X)$, $a+b=Int(Cl(a\cup
        b))$.
        \item $(\mathbf{OB}(X),\varnothing,X,\cap,+)$ is a dually dense sublattice of the Boolean algebra $RO(X)$.
    \end{enumerate}
    \item If any of the (equivalent) conditions (a),(b),(c) or (d)
    of 1.\ is fulfilled then:
    \begin{enumerate}
        \item $(\mathbf{OB}(X),\varnothing,X,\cap,+)$ is an $O$-extensional distributive
        lattice.
        \item $X$ is a semiregular space.
    \end{enumerate}
\end{enumerate}
\end{theorem}

The following is a corollary of the above theorem.

\begin{corollary}\label{ext:semireg2}
Let $X$ be a topological space, let $L=(L,0,1,+,\cdot)$ be a lattice
and let $h$ be an embedding of the lower  semi-lattice $(L,0,1,\cdot)$
into the lattice $O(X)$ of open sets of $X$. Suppose that the
set $\mathbf{OB}(X)=\{h(a): a\in L\}$ forms an open basis for the
topology of $X$. Then:
\begin{enumerate}
    \item The following conditions are equivalent:
    \begin{enumerate}
        \item $L$ is $O$-extensional.
        \item $\mathbf{OB}(X)\subseteq RO(X)$.
        \item For all $a,b\in L$, $h(a + b)=Int(Cl(h(a)\cup
        h(b)))$.
        \item $h$ is a  dense embedding of $L$ into the Boolean algebra $RO(X)$.
    \end{enumerate}
    \item If any of the (equivalent) conditions (a),(b),(c) or (d)
    of 1.\ is fulfilled then:
    \begin{enumerate}
        \item $L$ is a $O$-extensional distributive lattice.
        \item $X$ is a semiregular space.
    \end{enumerate}
\end{enumerate}
\end{corollary}

\begin{remark}\label{important conclusion} {\rm  (i)  Let $\underline{D}=(D,C,\widehat{C}, \ll)$ be an EDC-lattice. Corollary \ref{ext:semireg1} shows that if we want to represent $\underline{D}$ by a dually dense embedding $h$ into the contact algebra $RC(X)$ of some topological space $X$ such that the topology of $X$ to be determined by the set $\mathbf{CB}(X)=\{h(a): a\in D\}$ considered as a closed base for $X$ we must require that the lattice $\underline{D}$ is U-extensional, i.e. to satisfy the axiom (Ext $\widehat{O}$)  (extensionality of underlap).
If in addition we want to apply the good properties of Lemma \ref{dense} then we must assume that $h$ is also a C-separable embedding into $RC(X)$. But then Corollary \ref{limitation} implies that $\underline{D}$ must satisfy also the axioms (U-rich $\ll $)  and (U-rich $\widehat{C}$).

(ii) Similar to the above conclusion is the following. Corollary \ref{ext:semireg2} shows that if we want to represent $\underline{D}$ by a  dense embedding $h$ into the contact algebra $RO(X)$ of some topological space $X$ such that the topology of $X$ to be determined by the set $\mathbf{OB}(X)=\{h(a): a\in D\}$ considered as an open base for $X$ we must require that the lattice $\underline{D}$ is O-extensional, i.e. to satisfy the axiom (Ext O)  (extensionality of overlap).
If in addition we want to apply the good properties of Lemma \ref{dual dense} then we must assume that $h$ is also a $\widehat{C}$-separable embedding into $RO(X)$. But then Corollary \ref{limitation} implies that $\underline{D}$ must satisfy also the axioms (O-rich $\ll $)  and (O-rich $\widehat{C}$). $\square$ }
\end{remark}

\begin{definition}\label{U-rich and O-rich lattices} {\bf U-rich and O-rich EDC-lattices.}  Let $\underline{D}=(D,C,\widehat{C}, \ll)$ be an EDC-lattice. Then:

(i)  $\underline{D}$ is called U-rich EDC-lattice if it satisfies the axioms (Ext $\widehat{O}$),  (U-rich $\ll $) and (U-rich $\widehat{C}$).

(ii) $\underline{D}$ is called O-rich EDC-lattice if it satisfies the axioms (Ext O),  (O-rich $\ll $) and (O-rich $\widehat{C}$).

\end{definition}

 The aim of the next sections is to develop the topological representation theory of U-rich and O-rich EDC-lattices.
%%%%%%%%%%%%%%%%%%%%%%%%%%%%%%%%%%%%%%%%%%%%%%%%%%%%%%%%%%%%%%%%%%%%%%%%%%%%%%%%%%%%%%%%%%%%%%%%%%%%%%%%%%%%%%%%%%%%%%%%%%%%555
\section{Topological representation theory of   U-rich EDC-lattices}\label{Uext}
%%%%%%%%%%%%%%%%%%%%%%%%%%%%%%%%%%%%%%%%%%%%%%%%%%%%%%%%%%%%%%%%%%%%%%%%%%%%%%%%%%%%%%%%%%%%%%%%%
 The aim of this section is to develop a topological representation theory for U-rich EDC-latices. According to Theorem \ref{semiregularity1} we will look for a representation with regular closed sets. To realize this we will follow the  representation theory of contact algebras by regular closed sets developed in \cite{dv,V}, updating the results of Section 4 from \cite{dmvw2} to the case of U-rich EDC-lattices. We will consider also extensions of U-rich  EDC-lattices with some of the additional axioms mentioned in Section \ref{additional axioms}.  The scheme of the representation procedure   is the following: for each
U-rich  EDC-lattice  $D$ from a given
class, determined by the additional axioms, we will  do the following:

\begin{itemize}

\item  Define a set $X(D)$ of "abstract points" of $\underline{D}$,

\item define a topology in $X(D)$ by the set $\mathbf{CB}(X(D))=\{h(a): a\in D\}$, considered as a closed base of the topology, where $h$ is the intended embedding of Stone type: $h(a)=\{\Gamma: \Gamma$ is "abstract point" and $a\in \Gamma\}$. $X(D)$ is called the \emph{canonical topological space of $\underline{D}$} and $h$ is called \emph{canonical embedding},

\item establish that $h$ is a dual dense embedding of the lattice $\underline{D}$ into the Boolean algebra $RC(X(D))$ of regular closed sets of  the space $X(D)$.

 \end{itemize}

  We will consider separately the cases of representations in $T_{0}$, $T_{1}$ and $T_{2}$ spaces which requires introducing different "abstract points".

%%%%%%%%%%%%%%%%%%%%%%%%%%%%%%%%%%%%%%%%%%%%%%%%%%%%%%%%%%%%%%%%%%%%%%%%%%%%
\subsection{Representations in $T_{0}$ spaces}\label{T0repesentations}
%%%%%%%%%%%%%%%%%%%%%%%%%%%%%%%%%%%%%%%%%%%%%%%%%%%%%%%%%%%%%%%%%%%%%%%%%%%

Troughout this section we consider that $\underline{D}=(D,C, \widehat{C}, \ll)$ is a U-rich EDC-lattice.

\medskip

{\bf Abstract points of $\underline{D}$.}

\medskip

As in \cite{dmvw2}, we consider the abstract points of $\underline{D}$ to be clans (see \cite{dv} for the origin of this notion).  The definition is the following. A subset $\Gamma\subseteq D$ is a \emph{clan} if it satisfies the following conditions:

(Clan 1) $1\in \Gamma$, $0\not\in \Gamma$,

(Clan 2) If $a\in \Gamma$ and $a\leq b$, then $b\in \Gamma$,

(Clan 3) If $a+b\in \Gamma$, then $a\in \Gamma$ or $b\in \Gamma$,

(Clan 4) If $a,b\in \Gamma$ then $aCb$.

$\Gamma$ is a \emph{maximal clan} if it is maximal with respect to the set-inclusion. We denote by CLAN(D) (MaxCLAN(D) ) the set of all (maximal)  clans of $\underline{D}$.

The notion of clan is an abstraction  from the following natural example. Let $X$ be a topological space and $RC(X)$ be the contact algebra of regular-closed subsets of $X$ and let $x\in X$. Then the set $\Gamma_{x}=\{a\in RC(X):x\in a\}$ is a clan.

Now we will present a construction of clans which is similar to the constructions of clans in contact algebras. First we will introduce a new canonical relation between prime filters.

\begin{definition}\label{new canonical relation} Let $U,V$ be prime filters. Define a new  canonical relation  $R_{C}$  (\emph{$R_{C}$-canonical relation}) between prime filters as follows:

$UR_{C}V\leftrightarrow_{def} (\forall a\in U)(\forall b\in V)(aCb)$.

\end{definition}
Let us note that the  relation $R_{C}$ depends only on $C$ and can be defined also for filters. It is different from the canonical relation between prime filters defined in Section \ref{relrepresentation}, but  the presence of U-rich axioms makes it equivalent to $R^{c}$ as it can be seen from the following lemma.

\begin{lemma}\label{new canonical relation1}

(i) $R_{C}$ is reflexive and symmetric relation.

(ii) If $\underline{D}$ satisfies the axioms (U-rich $\ll $) and  (U-rich $\widehat{C}$) then $R_{C}=R^{c}$.

\end{lemma}

\begin{proof} (i) follows from the axioms (C4) and (C5).

(ii) The inclusion $R^{c}\subseteq R_{C}$ follows directly by the definition of $R^{c}$. For the converse inclusion suppose $UR_{C}V$. To show
$U R^{c} V$  we have to inspect the four cases of the definition of $R^{c}$.

\smallskip

\noindent {\bf Claim 1}: $a\in U$ and $b\in V$ implies $aCb$. This is just by the definition of $R_{C}$ .

\smallskip

\noindent{\bf Claim 2}: $a\in U$ and $b\not\in V$ implies $a\not\ll b$. For the sake of contradiction suppose $a\in U$ and $b\not\in V$ but $a\ll b$. Then by axiom (U-rich $\ll $) ( $a\ll b\rightarrow (\exists c)(b+c=1$ and $a\overline{C}c)$, we obtain $b+c=1$ and $a\overline{C}c$. Conditions $b+c=1$ and
$b\not\in V$ imply $c\in V$. But $a\in U$, so $aCc$ - a contradiction.

\smallskip

\noindent{\bf Claim 3:} $a\not\in U$ and $b\in V$ implies $b\not\ll a$. The proof is similar to the proof of Claim 2.

\smallskip

\noindent{\bf Claim 4:} $a\not\in U$ and $b\not\in V$ implies $a\widehat{C}b$. The proof is similar to the proof of Claim 2 by the use of axiom

(U-rich $\widehat{C}$) $a\overline{\widehat{C}}b\rightarrow    (\exists c, d)(a+c=1, b+c=1$ and $c\overline{C}d)$. $\square$

\end{proof}

The following statement lists some facts about the relation $R_{C}$.

 \begin{facts}\label{facts for R_{C}} \cite{DuV,dv,dmvw2}.
 \begin{enumerate}
  \item Let $F,G$ be filters and $FR_{C}G$ then there are prime filters $U,V$ such that $F\subseteq U$, $G\subseteq V$ and $UR_{C}V$.

\item  For all $a,b\in D$: $aCb$ iff there exist prime filters $U,V$ such that $UR_{C} V$, $a\in U$ and $b\in V$.

 \end{enumerate}
 \end{facts}

In the following lemma we list some facts about clans (see, for instance, \cite{dv,dmvw2}).

\begin{facts}\label{facts for clans}

\begin{enumerate}

\item  Every   prime filter is a clan.

\item  The complement of every clan is an ideal.

\item  If $\Gamma$ is a clan and $F$ is a filter such that $F\subseteq \Gamma$, then there is a prime filter $U$ such that $F\subseteq U\subseteq\Gamma$. In particular, if $a\in \Gamma$, then there exists a prime filter $U$ such that $a\in U\subseteq\Gamma$.

\item Every clan $\Gamma$ is the union of all prime filters contained in $\Gamma$.

\item Every clan is contained in a maximal clan.

\item Let $\Sigma$ be a  nonempty set of prime filters such that for every $U,V\in \Sigma$ we have $UR_{C}V$ and let $\Gamma$ be the union of the elements of $\Sigma$. Then $\Gamma$ is a clan and every clan can be obtained in this way.

\item Let $U,V$ be prime filters, $\Gamma$ be a clan  and $U,V\subseteq \Gamma$,. Then $UR_{C}V$ and $UR^{c}V$.
\end{enumerate}

\end{facts}

\begin{lemma}\label{special property of clans} Let $\Gamma$ be a clan and $a\in D$. Then the following two conditions are equivalent:

(i) $(\forall c\in D)(a+c=1\rightarrow c\in \Gamma)$,

(ii) There exists a prime filter $U\subseteq \Gamma$ such that $a\not\in U$.

\end{lemma}

\begin{proof} (i)$\rightarrow$ (ii). Suppose that (i) holds. It is easy to see that the set $F=\{c: a+c=1\}$ is a filter. The complement $\overline{\Gamma}$ of $\Gamma$ is an ideal (Facts \ref{facts for clans}) and hence $\overline{\Gamma}\oplus(a]$ is an ideal. We will show that $F\cap \overline{\Gamma}\oplus(a]=\varnothing$. Suppose the contrary. Then there is a $c$ such that $a+c=1$ (and hence by (i) $c\in \Gamma$) and $c\in \overline{\Gamma}\oplus(a]$. Then there is $x\in \overline{\Gamma}$ such that $c\leq x+a$. From here we get: $1=a+c\leq a+x+a=x+a$, hence $x+a=1$ and by (i) - $x\in \Gamma$, contrary to $x\in \overline{\Gamma}$. Now we can apply Filter-extension Lemma and obtain a prime filter $U$ extending $F$ such that $U\cap \overline{\Gamma}\oplus(a]=\varnothing$. It follows from here that $a\not\in U$, $U\cap \overline{\Gamma}=\varnothing$ which implies $U\subseteq \Gamma$.

(ii)$\rightarrow$(i). Suppose (ii) holds: $U\subseteq\Gamma$ and $a\not\in U$. Suppose $a+c=1$. Then $c\in U\subseteq\Gamma$, so $c\in \Gamma$ - (i) is fulfilled. $\square$

\end{proof}

\bigskip
%%%%%%%%%%%%%%%%%%%%%%%%%%%%%%%%%%%%%%%%5
{\bf Defining the canonical topological space $X(D)$ of $\underline{D}$ and the canonical  embedding $h$.}

\bigskip

Define the Stone like embedding: $h(a)=\{\Gamma\in CLAN(D): a\in \Gamma\}$ and consider   the set $\mathbf{CB}(X)=\{h(a): a\in D\}$ as a closed base of the topology in $X(D)=CLAN(D)$.

\begin{lemma}\label{T0embedding1} The space $X(D)$ is
semiregular  and $h$ is a dually dense  embedding of $\underline{D}$ into the
contact Boolean algebra $RC(X(D))$.
\end{lemma}

\begin{proof}  Using the
properties of clans, one can   easily  check  that
$h(0)=\varnothing$, $h(1)=X$, and that $h(a+b)=h(a)\cup h(b)$. This
shows that the set $\mathbf{CB}(X(D))=\{ h(a): a\in D\}$ is closed
under finite unions and, in fact, it is  a closed basis for the
topology of $X$. Also we have the implication: $a\leq b$ then $h(a)\subseteq h(b)$.

To show that $h$ is an embedding we use the fact that prime filters are clans and  prove that $a\not\leq b$ implies $h(a)\not\subseteq h(b)$. Indeed, from $a\not\leq b$ it follows by the theory of distributive lattices (see \cite{bd}) that there exists a prime filter U (which is also a clan) such that $a\in U$ (so $U\in h(a)$) and $b\not\in U$ (so, $U\not\in h(b)$), which proves that $h(a)\not\subseteq h(b)$.  Consequently, $h$ is an embedding of the upper semi-lattice $(D,0,1,+)$ into the lattice of closed sets of the space $X(D)$. By Corollary \ref{ext:semireg1}, $X(D)$ is a
semiregular space and $h$ is a dually dense embedding of $D$ into
the Boolean algebra $RC(X)$. It remains to show that $h$ preserves the relations $C, \widehat{C}$ and $\ll$. This follows from the following   claim.

%%%%%%%%%%%%%%%%%%%%%%%%%%5
\begin{claim}\label{preservation 1} (i) Let $\Gamma$ be a clan and $a\in D$.  Then following equivalence holds:

 $\Gamma\in h(a)$ iff there exists a prime filter $U$ such that $a\in U\subseteq \Gamma$.

(ii) Let $\Gamma$ be a clan and $a\in D$.  Then following conditions are  equivalent:

\begin{quote}
(I) $(\forall c\in D)(a+c=1\rightarrow c\in \Gamma)$,

(II)  $\Gamma\in Cl(-h(a))$,

(III)  There exists a prime filter $U$ such that $a\not\in U\subseteq \Gamma$.
\end{quote}

(iii) $aCb$ iff $h(a)\cap h(b)\not=\varnothing$,

(iv) $a\not\ll b$ iff $h(a)\cap Cl(-h(b))\not=\varnothing$.

(v) $a\widehat{C}b$ iff $Cl(-h(a))\cap Cl(-h(b))\not=\varnothing$,

\end{claim}

{\bf Proof of the claim.} (i) follows easily from Facts \ref{facts for clans} (3.).

(ii) The proof of $(I)\leftrightarrow (II)$ follows  by the following sequence of equivalences:

\smallskip
\begin{quote}
$(\forall c\in D)(a+c=1\rightarrow c\in \Gamma)$ iff

$(\forall c\in D)(h(a)\cup h(c)=X(D)\rightarrow \Gamma\in h(c))$ iff

$(\forall c\in D)(-h(a)\subseteq h(c)\rightarrow \Gamma\in h(c))$ iff

$\Gamma\in Cl(-h(a))$

\end{quote}

The first equivalence holds because $h$ is an embedding of the upper semi-lattice  $(D,0,1,+)$ into the lattice of closed sets of the space $X(D)$, the third equivalence uses the fact that the set $\{h(c):c\in D\}$ is a closed base of the topology of $X(D)$.

The equivalence $(I)\leftrightarrow(III)$ is just the Lemma \ref{special property of clans}.
\smallskip

(iii) ($\Rightarrow$ ) Suppose $aCb$, then by Lemma \ref{C} (i) there exist prime filters $U$, and $V$ such that $UR^{c}V$, $a\in U$ and $b\in V$. Let $\Gamma=U\cup V$. By Facts \ref{facts for clans}  $\Gamma$ is a clan, obviously containing $a$ and $b$, which implies $h(a)\cap h(b)\not=\varnothing$.

\smallskip

($\Leftarrow$) Suppose $h(a)\cap h(b)\not=\varnothing$. Then there exists a clan $\Gamma$ containing $a$ and $b$,  hence $aCb$.

(iv) ($\Rightarrow$ ) Suppose  $a\not\ll b$. Then by Lemma \ref{C} (ii) there exist prime filters $U,V$ such that $UR^{c}V$, $a\in U$ and $b\not\in V$. Let $\Gamma=U\cup V$, then $\Gamma$ is a clan containing $U$ and $V$. So, $a\in \Gamma$ and hence $\Gamma\in h(a)$. From the condition $b\not\in V\subseteq \Gamma$ we obtain by (ii) that $\Gamma\in Cl(-h(b))$ and hence $h(a)\cap Cl(-h(b))\not=\varnothing$.

$(\Leftarrow$) Suppose $h(a)\cap Cl(-h(b))\not=\varnothing$. Then there exists a clan $\Gamma\in h(a)$ and $\Gamma\in Cl(-h(b))$. It follows by (i) that there exists a prime filter $U$ such that $a\in U\subseteq\Gamma$ and by (ii) we obtain that there exists a prime filter $V$ such that $b\not\in V\subseteq \Gamma$. Condition   $U,V\subseteq\Gamma$ implies by Facts \ref{facts for clans} (7.) that $UR^{c}V$. Using the properties of the relation $R^{c}$ and $a\in U$ and $b\not\in V$ we get $a\not \ll b$.

(v) The proof of (v) is similar to the proof of (iv) with the use of Lemma \ref{dual C}.
This finishes the proof of Lemma \ref{T0embedding1} $\square$
\end{proof}

\begin{lemma}\label{$T_0$compact1}

The following conditions are true for the canonical space  $X(D)$:

 (i) $X(D)$ is $T_0$.

 (ii) $X(D)$ is compact.
\end{lemma}
\begin{proof} The proof is the same as the proof of Lemma 19 from \cite{dmvw2}.$\square$

\end{proof}

\begin{lemma}\label{C-separability} The mapping $h$ is a C-separable embedding
of $D$ into $RC(X(D))$.
\end{lemma}

\begin{proof} This lemma was proved in \cite{dmvw2}  by a special construction. Since the definition of C-separability for EDC-lattices uses an extended definition for which the special construction from \cite{dmvw2} does not hold, in this paper we give a new proof deducing  the statement from the compactness of the space $X(D)$.

We have to proof the following three statements, corresponding to the three clauses of the condition of C-separability (see Definition \ref{C-separability1}).

\smallskip

\noindent (C-separability for C) $(\forall \alpha, \beta\in RC(X(D)))(\alpha\cap\beta=\varnothing\rightarrow(\exists a,b\in D)(\alpha\subseteq h(a), \beta\subseteq h(b), h(a)\cap h(b))=\varnothing$.

\smallskip

\noindent (C-separability for $\widehat{C}$)    $(\forall \alpha, \beta \in RC(X(D))( Cl(-\alpha)\cap Cl(-\beta)=\varnothing \rightarrow (\exists a, b \in D)(\alpha\cup h(a)= X(D), \beta \cup h(b)=X(D), h(a)\cap h(b)=\varnothing)$.

\smallskip

\noindent (C-separability for $\ll$) $(\forall \alpha, \beta \in RC(X(D))(\alpha\cap Cl(-\beta)=\varnothing\rightarrow (\exists a, b\in D)(\alpha\subseteq h(a), \beta\cup h(b)=X(D), h(a)\cap h(b)=\varnothing)$.

As an example we shall prove  the condition (C-separability for C). The proofs for the other two conditions are similar.

\smallskip

\textbf{Proof of (C-separability for C)}. Let $\alpha, \beta\in RC(X(D))$   and $\alpha\cap\beta=\varnothing$.  Since $\alpha$ and $\beta$ are closed sets they can be represented as intersections from the elements of the basis $\mathbf{CB}(X(D))=\{ h(c): c\in D\}$ of $X(D)$. So there are subsets $A, B\subseteq\mathbf{CB}(X(D))$ such that $\alpha=\bigcap\{h(c): h(c)\in A\}$ and $\beta=\bigcap\{h(c): h(c)\in B\}$. Then $\alpha\cap\beta=\bigcap\{h(c): h(c)\in A\}\cap \bigcap\{h(c): h(c)\in B\}=\varnothing$. By the compactness of $X(D)$ (Lemma \ref{$T_0$compact1} (ii)), there are finite subsets $A_{0}\subseteq A$ and $B_{0}\subseteq B$ such that $\alpha\cap\beta=\bigcap\{h(c): h(c)\in A_{0}\}\cap \bigcap\{h(c): h(c)\in B_{0}\}=\varnothing$. Let $A_{0}=\{h(c_{1}),..., h(c_{n})\}$ and
$B_{0}=\{h(d_{1}),..., h(d_{m})\}$ and let $a=c_{1}\cdot...\cdot c_{n}$ and $b=d_{1}\cdot...\cdot d_{m}$. Then $h(a)\subseteq h(c_{i})$, $i=1...n$ and from here we get $h(a)\subseteq h(c_{1})\cap...\cap h(c_{n})$. Analogously we obtain that $h(b)\subseteq h(d_{1})\cap...\cap h(d_{m})$. Consequently $h(a)\cap h(b)\subseteq (h(c_{1})\cap...\cap h(c_{n})\cap(h(d_{1})\cap...\cap h(d_{m}))=\varnothing$, so  $h(a)\cap h(b)=\varnothing$. Also we have $\alpha\subseteq h(c)$ for all $h(c)\in A$ and consequently for all $h(c)\in A_{0}$. Hence $\alpha\subseteq h(c_{1})\cdot...\cdot h(c_{n})=h(c_{1}\cdot...\cdot c_{n})=h(a)$, so $\alpha\subseteq h(a)$. Analogously we get $\beta\subseteq h(b)$. $\square$
\end{proof}

The following theorem is the main result of this section.

\begin{theorem}\label{rep:ext1} {\bf Topological representation theorem for
$U$-rich EDC-lattices}\\ Let
$\underline{D}=(D, C, \widehat{C}, \ll) $ be an $U$-rich EDC-lattice. Then there exists a compact semiregular $T_0$-space $X$ and
a dually dense and $C$-separable embedding $h$ of $\underline{D}$ into the
Boolean contact algebra $RC(X)$ of the regular closed sets of $X$.
Moreover:

(i) $\underline{D}$ satisfies (Ext C) iff  $RC(X)$ satisfies (Ext C); in
this case $X$ is weakly  regular.

(ii) $\underline{D}$ satisfies (Con C) iff  $RC(X)$ satisfies (Con C); in this
case $X$ is connected.

(iii) $\underline{D}$ satisfies (Nor 1) iff  $RC(X)$ satisfies (Nor 1); in this
case $X$ is $\kappa$-normal.
\end{theorem}

\begin{proof} Let $X$ be the canonical space $X(D)$ of $\underline{D}$ and $h$ be the
canonical embedding of $\underline{D}$. Then, the theorem is a corollary of
Lemma \ref{T0embedding1}, Lemma \ref{$T_0$compact1}, Lemma
\ref{C-separability} and Lemma \ref{topoequivalents}.$\square$
\end{proof}

Note that Theorem \ref{rep:ext1} generalizes several results from \cite{dv,dw} to the distributive case.

%%%%%%%%%%%%end%%%%%%%%%%%%%%%%%%%%%%%%%%%%%%%%%%%%%%%%%%%
\subsection{Representations in $T_1$  spaces}\label{T1 space}
%%%%%%%%%%%%%%%%%%%%%%%%%%%%%%%%%%%%%%%%%%%%%%%%%%%%%%%%
The aim of this section is to obtain representations of some U-rich EDC-lattices in $T_1$-spaces extending the corresponding results from \cite{dmvw2}. The constructions will be slight modifications of the corresponding constructions from the previous section, so we will  be sketchy.

Let $\underline{D}=(D, C, \widehat{C}, \ll) $ be an $U$-rich EDC-lattice.
In the previous section the abstract points were clans and this
guarantees that the representation space is $T_0$. To obtain representations in $T_{1}$ spaces we assume abstract points to be maximal clans, so for  the canonical space of $\underline{D}$ we put $X(D)=MaxCLAN(D)$ and define the canonical embedding $h$ to be $h(a)=\{\Gamma\in MaxCLAN(D):a\in \Gamma\}$. The topology in $X(D)$ is defined considering the set $\mathbf{CB}(X(D))=\{h(a): a\in D\}$ to be the closed base for the space. Note that in general, without additional axioms we can not prove in this case that $h$ is an embedding. In order to guarantee this we will assume that $\underline{D}$ satisfies additionally the axiom of C-extensionality

\smallskip

(Ext C) $a\not=1\rightarrow(\exists b\not=0)(a\overline{C}b)$.

\smallskip

\noindent Note that in this case, due to U-extensionality (see Section \ref{additional axioms}), the lattice $\underline{D}$ satisfies also the axiom

\smallskip

(EXT C) $a\not\leq b\rightarrow (\exists c)(aCc$ and $b\overline{C}c)$,

\smallskip

\noindent which is essential in the proof that $h$ is an embedding.

\begin{lemma}\label{T1embedding} The space $X(D)$ is a
semiregular  and $h$ is a dually dense  embedding of $D$ into the
contact Boolean algebra $RC(X(D))$.
\end{lemma}
\begin{proof} The proof is  similar to the proof of Lemma
\ref{T0embedding1}, so we will indicate only the differences. First we show that $h$ is  an embedding of the upper semi-lattice $(D,0,1,+)$ into the lattice of closed sets of the space $X(D)$. The only new thing which we have to show is: If $a\not\leq b$ then $h(a)\not \subseteq h(b)$. To do this suppose  $a\not\leq b$. Then by axiom (EXT C) there exists $c\in D$ such that $aCc$ but $b\overline{C}c$. Condition $aCc$ implies that there exist prime filters $U,V$ such that $UR^{c}V$, $a\in U$ and $c\in V$. Let $\Gamma_{0}=U\cup V$. $\Gamma_{0}$ is a clan and by Facts \ref{facts for clans} it is contained in a maximal clan $\Gamma$. Obviously $a,c\in\Gamma$, so $\Gamma\in h(a)$. But $b\overline{C}c$ implies that $b\not\in \Gamma$ (otherwise we will get $bCc$). Conditions $\Gamma\in h(a)$ and $\Gamma\not\in h(b)$ show that $h(a)\not\subseteq h(b)$. Thus, by Corollary \ref{ext:semireg1} $h$ is a dually dense embedding of $\underline{D}$ into the Boolean algebra $RC(X(D))$. It remains to show that $h$ preserves the relations $C, \widehat{C}$ and $\ll$. The proof is almost the same as in the corresponding proof of Lemma \ref{T0embedding1}. The only new thing is when we construct a certain clan from prime filters satisfying the relation $UR^{c}V$ in the form $U\cup V$, then we extend it into a maximal clan. Note also that Claim \ref{preservation 1} remains true. We demonstrate this by considering only the preservation of $\ll$. We have to show:

\smallskip $a\not\ll b$ iff $h(a)\cap Cl(-h(b)\not=\varnothing$

($\Rightarrow$) Suppose $a\not\ll b$. Then by Lemma \ref{C}  $(\exists U, V\in PF(D))(a\in U$ and $b\not\in V$ and $U R^{c}V)$. Define $\Gamma_{0}=U\cup V$. $\Gamma_{0}$ is a clan containing $U$ and $V$. Extend $\Gamma_{0}$ into a maximal clan $\Gamma$. Then $\Gamma$ contains $a$, so $\Gamma\in h(a)$. We have also that $b\not\in V\subseteq\Gamma$, so by the Claim \ref{preservation 1} $\Gamma\in Cl(-h(b))$.
\smallskip

($\Leftarrow)$ The proof is identical to the corresponding proof from Lemma \ref{T0embedding1}. $\square$
\end{proof}

\begin{lemma}\label{T1compact1} The space $X(D)$ satisfies the following conditions:

(i) $X(D)$ is $T_{1}$,

(ii) $X(D)$ is compact,

(iii) $h$ is C-separable embedding.

\end{lemma}

\begin{proof} (i) Let $\Gamma$ be an arbitrary maximal clan. The space $X(D)$ is $T_{1}$ iff the singleton set  $\{ \Gamma\}$ is closed, i.e. $Cl(\{\Gamma\})=\{\Gamma\}$. This follows  by the maximality of $\Gamma$ as follows. Let $\Delta$ be a maximal clan. Then:

%\begin{quote}
\noindent $\Delta\in Cl(\{\Gamma\})$ iff
$(\forall c\in D)(\{\Gamma\}\subseteq h(c)\rightarrow \Delta\in h(c))$ iff
$(\forall c\in D)(\Gamma\in h(c)\rightarrow \Delta\in h(c))$ iff
$(\forall c\in D)(c\in \Gamma\rightarrow c\in \Delta$ iff $\Gamma\subseteq\Delta)$ iff
$\Gamma=\Delta$ iff $\Delta\in \{\Gamma\}$.

 This chain shows that  indeed $Cl(\{\Gamma\})=\{\Gamma\}$.

 (ii) The proof is similar to the proof of Lemma \ref{$T_0$compact1} (ii)
 %\end{quote}

(iii) follows from (ii) as in the proof of Lemma \ref{C-separability}. $\square$
\end{proof}

 \begin{theorem}\label{rep:ext2} {\bf Topological representation theorem for C-extensional
$U$-rich EDC-lattices} Let
$\underline{D}=(D, C, \widehat{C}, \ll) $ be a C-extensional $U$-rich EDC-lattice. Then there exists a compact weakly regular $T_1$-space $X$ and
a dually dense and $C$-separable embedding $h$ of $\underline{D}$ into the
Boolean contact algebra $RC(X)$ of the regular closed sets of $X$.
Moreover:

(i) $\underline{D}$ satisfies (Con C) iff  $RC(X)$ satisfies (Con C); in this
case $X$ is connected.

(ii) $\underline{D}$ satisfies (Nor 1) iff  $RC(X)$ satisfies (Nor 1); in this
case $X$ is $\kappa$-normal.

\end{theorem}

\begin{proof} The proof follows from Lemma \ref{T1embedding}, Lemma \ref{T1compact1} and Lemma \ref{topoequivalents}.

\end{proof}

%%%%%%%%%%%%end%%%%%%%%%%%%%%%%%%%%%%%%%%%%%%%%%%%%%%%%%%%
\subsection{Representations in $T_2$  spaces}
%%%%%%%%%%%%%%%%%%%%%%%%%%%%%%%%%%%%%%%%%%%%%%%%%%%%%%%%
In the previous section we proved representability in $T_{1}$ spaces of U-rich EDC-lattices satisfying the axiom of C-extensionality (Ext C). The $T_{1}$ property of the topological space was guaranteed by the fact that abstract points are maximal clans. In this section we will show that adding the axiom (Nor 1) we can obtain representability in compact $T_{2}$-spaces. The reason for this is that the axiom (Nor 1) makes possible to use new abstract points - the so called clusters, which are maximal clans satisfying some  additional properties yielding $T_{2}$ separability of the topological space. Clusters have been used in the compactification theory of proximity spaces (see more about their origin in \cite{Thron}). They have been adapted in algebraic form in the representation theory of contact algebras in \cite{dv,VDDB}. In \cite{dmvw2} their definition and some constructions are modified for the distributive case.   We remaind below the corresponding definition.

\begin{definition}\label{clusters} Let $\underline{D}=(D, C, \widehat{C}, \ll) $  be an EDC-lattice. A clan $\Gamma$ in $\underline{D}$ is called a cluster if it satisfies the following condition:

(Cluster)  If for all $b\in \Gamma$ we have $aCb$, then $a\in \Gamma$.

We denote the set of clusters in $\underline{D}$ by $CLUSTER(D)$.

\end{definition}

Let us note that not in all EDC-lattices there are clusters.  The following lemma shows that the axiom (Nor 1) guarantees existence of clusters and some important properties needed for the representation theorem.

\begin{lemma} \label{facts for clusters} \cite{dmvw2} Let $\underline{D}=(D, C, \widehat{C}, \ll) $ be an EDC-lattice. Then:

 (i) Every  cluster is a maximal clan.

(ii) If $\underline{D}$ satisfies (Nor 1) then every maximal clan is a cluster.

(iii) If $\Gamma$ and $\Delta$ are  clusters such that $\Gamma\not=\Delta$, then there are $a\not\in \Gamma$ and $b\not\in \Delta$ such that $a+b=1$.

\end{lemma}

To build the canonical space $X(D)$  we assume in this section that $\underline{D}=(D, C, \widehat{C}, \ll) $ is an U-rich EDC-lattice satisfying the axioms (Ext C) and (Nor 1). We define  $X(D)=CLUSTER(D)$,  $h(a)=\{\Gamma\in CLUSTER(D): a\in \Gamma\}$ and define the topology in $X(D)$ considering the set $\mathbf{CB}(X)=\{h(a): a\in D\}$ as a basis for closed sets in $X(D)$. Since the points of $X(D)$ are maximal clans, just as in Section \ref{T1 space} we can prove the following lemma.

\begin{lemma}\label{T2embedding}  The space $X(D)$ is a
semiregular  and $h$ is a dually dense  embedding of $D$ into the
contact Boolean algebra $RC(X(D))$.
\end{lemma}

\begin{lemma}\label{T2compact}
(i) $X(D)$ is $T_{2}$,

(ii) $X(D)$ is compact,

(iii) $h$ is C-separable embedding.

\end{lemma}

\begin{proof} (i) To show that the space $X(D)$ is $T_{2}$ suppose that $\Gamma,\Delta$ are two different clusters. We have to find two disjoint open sets $A,B$ such that $\Gamma\in A$ and $\Delta \in B$. By Lemma \ref{facts for clusters} (iii) there are $a,b\in D$ such that $a\not\in \Gamma$ and $b\not\in \Delta$ such that $a+b=1$. Then by Lemma \ref{T2embedding} we get $\Gamma\not\in h(a)$, $\Delta\not\in h(b)$ and $h(a)\cup h(b)=X(D)$, hence $-h(a)\cap -h(b)=\varnothing$. Define $A=-h(a)$, $B=-h(b)$. Since $h(a)$ and $h(b)$ are closed sets, then $A$ and $B$ are open sets which separate the abstract points $\Gamma$ and $\Delta$.

The proof of (ii) and (iii) is the same as the proof of (ii) and (iii) in Lemma \ref{T1compact1}. $\square$
\end{proof}

\begin{theorem}\label{T2rep} {\bf Topological representation theorem for
$U$-rich EDC-lattices satisfying (Ext C) and (Nor 1).} Let
$\underline{D}=(D, C, \widehat{C}, \ll) $ be an  $U$-rich EDC-lattice satisfying (Ext C) and (Nor 1). Then there exists a compact  $T_2$-space $X$ and
a dually dense and $C$-separable embedding $h$ of $\underline{D}$ into the
Boolean contact algebra $RC(X)$ of the regular closed sets of $X$.
Moreover $\underline{D}$ satisfies (Con C) iff  $RC(X)$ satisfies (Con C) and in this case  $X$ is connected.
\end{theorem}

\begin{proof} The proof follows from Lemma \ref{T2embedding}, Lemma \ref{T2compact} and \ref{topoequivalents}. $\square$

\end{proof}

Let us note that this theorem generalizes several theorems from \cite{dv,V,deVries,VDDB}

%%%%%%%%%%%%%%%%%%%%%%%%%%%%%%%%%%%%%%%%%%%%%%%%%%%%%%%%%%%%%%%%%%%%%%%%%%%%%%%%%%%%%%%%%%%%%%%%
\section{Topological representation theory of   O-rich\\ EDC-lattices}
%%%%%%%%%%%%%%%%%%%%%%%%%%%%%%%%%%%%%%%%%%%%%%%%%%%%%%%%%%%%%%%%%%%%%%%%%%%%%%%%%%%%%%%%%

This section is devoted to the theory of dense  representations  for O-rich EDC-latices (see Definition  \ref{U-rich and O-rich lattices}). According to Theorem \ref{semiregularity} we will look for  dense representations with regular open sets. This case is completely dual to the corresponding theory developed in Section \ref{Uext}. For this reason we will only sketch  the main representation scheme and the  definitions of abstract points for the $T_{0}$, $T_{1}$ and $T_{2}$ representations.

The representation scheme is dual to the scheme presented in Section \ref{Uext}:

 \begin{itemize}

\item  Define a set $X(D)$ of "abstract points" of $\underline{D}$,

\item define a topology in $X(D)$ by the set $\mathbf{OB}(X(D))=\{h(a): a\in D\}$, considered as an open base of the topology, where $h$ is the intended embedding of Stone type: $h(a)=\{\Gamma: \Gamma$ is "abstract point" and $a\in \Gamma\}$. $X(D)$ is called the \emph{canonical topological space of $\underline{D}$} and $h$ is called \emph{canonical embedding},

\item establish that $h$ is a  dense embedding of the lattice $\underline{D}$ into the Boolean algebra $RO(X(D))$ of regular open sets of  the space $X(D)$.

 \end{itemize}

For the case of $T_{0}$ dense representation we consider a notion of abstract  point which is dual to the notion of clan. This is the so called E-filter (Efremovich filter). E-filters were used in the theory of proximity spaces (see \cite{Thron}). In the context of contact algebras they were introduced for the first time in \cite{dv}. The definition adapted for the language of EDC-lattices is the following.

\begin{definition}\label{E-filter} Let $\underline{D}=(D, C, \widehat{C}, \ll)$ be an EDC-lattice. A subset $\Gamma\subseteq D$ is called an \emph{E-filter} if it  satisfies the following properties:

(E-fil 1)  $\Gamma$ is a proper filter in $\underline{D}$, i.e. $0\not\in \Gamma$,

(E-fil 2)  If $a\not\in \Gamma$ and $b\not\in \Gamma$, then $a\widehat{C}b$.

$\Gamma$ is a minimal E-filter if it is minimal in the set of all E-filters of $\underline{D}$ with respect to set inclusion.

\end{definition}

This definition comes as an abstraction  from the following natural example. Let $X$ be a topological space, $x\in X$ and $RO(X)$ be the set of all regular-open sets of $X$. Then the set $\Gamma_{x}=\{a\in RO(X): x\in a\}$ is an E-filter in the contact algebra $RO(X)$. Note that the definition of E-filter is based not on the relation of contact $C$, but on the dual contact $\widehat{C}$.

A general construction of E-filters can be obtain dualizing the construction of clans from Section \ref{T0repesentations}. Just to show how this dual construction goes on and how the O-rich axioms works, we will repeat some steps omitting the proofs.

 First we will introduce a new canonical relation between prime filters.

\begin{definition} Let $U,V$ be prime ideals. Define a new  canonical relation  $\widehat{R}_{\widehat{C}}$  (\emph{$\widehat{R}_{\widehat{C}}$-canonical relation}) between prime ideals as follows:

$U\widehat{R}_{\widehat{C}}V\leftrightarrow_{def} (\forall a\in U)(\forall b\in V)(a\widehat{C}b)$.

If $U,V$ are prime filters then we define $UR_{\widehat{C}}V\leftrightarrow_{def} \overline{U}\widehat{R}_{\widehat{C}}\overline{V}$.

\end{definition}

Let us note that the  relation $\widehat{R}_{\widehat{C}}$ depend only on $\widehat{C}$ and can be defined also for ideals. It is different from the canonical relation $\widehat{R}^{c}$ between prime ideals defined in Section \ref{relrepresentation}, but  the presence of O-rich axioms  makes it equivalent to $\widehat{R}^{c}$ as it is stated in the following lemma.

\begin{lemma}

(i) $\widehat{R}_{\widehat{C}}$ is a reflexive and symmetric relation.

(ii) If $\underline{D}$ satisfies the axioms (O-rich $\ll $) and  (O-rich $\widehat{C})$, then $\widehat{R}_{\widehat{C}}=\widehat{R}^{c}$.

\end{lemma}

The following statement lists some facts about the relation $R_{C}$.

 \begin{facts}
 \begin{enumerate}
  \item Let $F,G$ be ideals and $F\widehat{R}_{\widehat{C}}G$ then there are prime ideals $U,V$ such that $F\subseteq U$, $G\subseteq V$ and $U\widehat{R}_{\widehat{C}}V$.

\item  For all $a,b\in D$: $a\widehat{C}b$ iff there exist prime ideals $U,V$ such that $U\widehat{R}_{\widehat{C}} V$, $a\in U$ and $b\in V$.

\item For all $a,b\in D$: $a\widehat{C}b$ iff there exist prime filters $U,V$ such that $UR_{\widehat{C}} V$, $a\not\in U$ and $b\not\in V$.

 \end{enumerate}
 \end{facts}

In the following lemma we list some facts about E-filters.

\begin{facts}\label{facts for E-filters}

\begin{enumerate}

\item  Every   prime filter is an E-filter.

\item If $\Gamma$ is an E-filter and $a\not\in \Gamma$, then there exists a prime filter $U$ such that $\Gamma\subseteq U$ and $a\not\in U$.

\item  Every  E-filter   $\Gamma$ is the intersection of all prime filters containing $\Gamma$.

\item Every E-filter contains a minimal E-filter.

\item Let $\Sigma$ be a nonempty set of prime filters such that for every $U,V\in \Sigma$ we have $UR_{\widehat{C}}V$ and let $\Gamma$ be the intersection  of the elements of $\Sigma$. Then $\Gamma$ is an  E-filter and every E-filter can be obtained in this way.

\item Let $U,V$ be prime filters, $\Gamma$ be an E-filter,  $\Gamma\subseteq U$ and $\Gamma\subseteq V$. Then $UR_{\widehat{C}}V$ and $UR^{c}V$.
\end{enumerate}

\end{facts}

Using the above facts one can prove the following representation theorem.

\begin{theorem} Let $\underline{D}=(D, C, \widehat{C}, \ll)$ be an O-rich  EDC-lattice. Then there exists a compact semi-regular space $X$ and a dense and $\widehat{C}$-separable embedding $h$ from $\underline{D}$ into the contact algebra $RO(X)$ of regular-open sets of $X$. Moreover:

(i) If $\underline{D}$ satisfies (Ext $\widehat{C}$), then $X$ is weakly regular,

(ii) If $\underline{D}$ satisfies (Con $\widehat{C}$), then $X$ is a connected space,

(iii) If $\underline{D}$ satisfies (Nor 2), then $X$ is $\kappa$-normal.

\end{theorem}

Abstract points for dense representations in $T_{1}$ spaces are minimal E-filters and abstract points for dense representations in $T_{2}$ spaces are duals of clusters introduced in \cite{dv} under the name \emph{co-clusters}. We adapt this notion for the language of EDC-lattices as follows:

\begin{definition} An E-filter $\Gamma$ is called co-cluster if it satisfies the following condition:

(Co-cluster) If  $(\forall b\not\in \Gamma)(a\widehat{C}b)$, then $a\not\in \Gamma$. (or, equivalently, if $a\in \Gamma$, then $(\exists b\not\in \Gamma)(a\overline{\widehat{C}}b))$.

\end{definition}

Let us show, for instance, the following statement for co-clusters, which is dual to the corresponding property for clusters as  maximal clans:

 \begin{lemma}
 Every co-cluster is a minimal E-filter.

\end{lemma}

\begin{proof} Suppose that $\Gamma$ is a co-cluster which is not a minimal E-filter. Then there exists an E-filter $\Delta$ such that $\Delta\subset\Gamma$, so $a\in \Gamma$ and $a\not\in \Delta$ for some $a$. Then there exists $b\not\in\Gamma $ such that $a\overline{\widehat{C}}b$. From here we get $b\in \Delta$. Consequently $b\in\Gamma$ - a contradiction.$\square$

\end{proof}

We left to the reader to formulate and proof the dual analog of Theorem \ref{rep:ext2} and Theorem \ref{T2rep}.

%%%%%%%%%%%%%%%%%%%%%%%%%%%%%%%%%%%%%%%%%%%%
\section{Concluding remarks}
%%%%%%%%%%%%%%%%%%%%%%%%%%%%%%%%%%%%%%%%%%%%%%%%

In this paper we generalized the notion of contact algebra by weakening the algebraic part to distributive lattice. One solution of this problem   was given  in \cite{dmvw2} including in the definition only the contact relation. However, the obtained axiomatization in \cite{dmvw2} is in a sense "incomplete", because it does not contain the definable in the Boolean case mereotopological relations of dual contact $\widehat{C}$ and non-tangential inclusion $\ll$ and its dual  $\gg$ and in this sense  the system is not closed under duality. We succeed in this paper to axiomatize all these relations considered as primitives  on the base of distributive lattices by means of universal first-order axioms. The resulting   system  is called "extended distributive contact lattice" (EDC-lattice). In this way we obtain, among others, the following two results. First, EDC-lattice is closed under duality,  and second,  it can be considered as an axiomatization of the universal fragment of contact algebras in the language of distributive lattices with the relations $C$, $\widehat{C}$ and $\ll$. We developed topological representation theory of EDC-lattices by means of regular closed and regular open sets generalizing in a quite non-trivial way the corresponding representation theory for contact algebras. Considering this representation theory on a weaker lattice  base   provided a deeper insight into  the interaction  of some notions taking place in the representation, which cannot be seen in the Boolean case. For instance we show the role of extensionality of underlap and overlap relations in case of dual dense and dense embeddings.

Our future plans include building of new  logics for qualitative spatial representation and reasoning based on EDC-lattices, studying the standard logical problems related to them: axiomatizability, decidability or undecidability, complexity. A good source for possible  generalizations and extensions is the paper \cite{BTV} containing many examples of spatial logics based on contact and precontact algebras.

\section{Acknowledgement}
The final publication is available at Springer via http://dx.doi.org/10.1007/s10472-016-9499-5.


\begin{thebibliography}{99}

\bibitem{A}
  M. Aiello, I. Pratt-Hartmann and J. van Benthem (eds.),
  \newblock  {\em Handbook of spatial logics},
   \newblock Springer, 2007.


\bibitem{bd}
R. Balbes and P. Dwinger. (1974).
\newblock {\em Distributive Lattices}.
\newblock University of Missouri Press, Columbia.

\bibitem{BTV}
 Ph. Balbiani, T. Tinchev and D. Vakarelov, Modal Logics for
the Region-based Theory of Space. {\it Fundamenta Informaticae},
Special Issue: Topics in Logic, Philosophy and Foundation of
Mathematics and Computer Science in Recognition of Professor
Andrzej Grzegorczyk, vol.   (81), (1-3), (2007), 29-82.


\bibitem{BD}
B. Bennett and  I. D\"{u}ntsch, Axioms, Algebras and Topology. In:
{\it Handbook   of Spatial Logics}, M. Aiello, I. Pratt, and J.
van Benthem (Eds.), Springer, 2007, 99-160.





\bibitem{CR}
A. Cohn and J. Renz. Qualitative spatial representation and
reasoning. In: F. van Hermelen, V. Lifschitz and B. Porter (Eds.)
{\it Handbook of Knowledge Representation}, Elsevier, 2008,
551-596.

\bibitem{dv}
G. Dimov and D. Vakarelov.
\newblock Contact algebras and region--based theory of space: {A} proximity
  approach I.
\newblock {\em Fundamenta Informaticae}, vol. 74, No 2-3,(2006)
209-249



\bibitem{DuV}
I. D{\"u}ntsch  and D. Vakarelov,
\newblock Region-based theory of discrette spaces: A proximity
approach.
\newblock In: Nadif, M., Napoli, A., SanJuan, E., and Sigayret, A.
EDS,
\newblock {\em Proceedings of Fourth International Conference
Journ{\'e}es de l'informatique Messine}, 123-129,
\newblock Metz, France, 2003.
\newblock Journal version in: {\em Annals of Mathematics and
Artificial Intelligence}, \textbf{ 49}(1-4):5-14, 2007.


\bibitem{dmvw1}
I. D{\"u}ntsch, W. MacCaull. D. Vakarelov  and M. Winter
\newblock Topological Representation of Contact Lattices.
\newblock {\em Lecture Notes in Computer Science} vol. 4136
(2006), 135-147.

\bibitem{dmvw2}
I. D{\"u}ntsch, I. MacCaull W. Vakarelov D. and Winter M.(2008)
Distributive contact lattices: Topological representation.
\newblock \emph{The Journal of logic and Algebraic Programming} 76 (2008), 18-34.


\bibitem{dw}
D{\"u}ntsch and M. Winter, M.
\newblock A representation theorem for {B}oolean contact algebras.
\newblock {\em Theoretical Computer Science (B)}, 347 (2005), 498-512.

\bibitem{EF}

M. Egenhofer and R. Franzosa, Point-set topological spatial relations. {\it Int. J. Geogr. Inform. Systems}, 5:161-174, 1991.

\bibitem{E}
R. Engelking,
\newblock General topology, PWN, 1977.


\bibitem{F}
P. Forrest (2010)
\newblock \emph{Mereotopology without Mereology},
J. Philos. Logic (2010) 39:229-254.

\bibitem{HG}
T.  Hahmann and M. Gr\"{o}uninger.
 \newblock  \emph{Region-based Theories of Space: Mereotopology
and Beyond}, in Qualitative Spatio-Temporal Representation and Reasoning: Trends
and Future Directions, edited by Hazarika, S., IGI Publishing, 2010. 2, 22

\bibitem{HWG}
T. Hahmann, M. Winter, M. Gruninger: Stonian p-Ortholattices: A new approach to the mereotopology RT0.
\newblock \emph{Artificial Intelligence}, 173:1424-1440, 2009.





\bibitem{Hazarika}
\emph{Qualitative Spatio-Temporal Representation
and Reasoning:
Trends and Future Directions}
Shyamanta M. Hazarika (Ed.), IGI Global; 1 edition (May 31, 2012)







\bibitem{NV}
Y. Nenov, D. Vakarelov,
\newblock \emph{Modal logics for mereotopological relations},
Advances in Modal Logic, volume 7, College Publications 2008, 249-272

\bibitem{RCC}
D. A.  Randell, Z.  Cui, and A. G. Cohn,
\newblock A spatial logic based on regions and connection.
\newblock In: B. Nebel, W. Swartout, C. Rich (EDS.) {\em Proceedings
of the 3rd International Conference Knowledge Representation and
Reasoning}, Morgan Kaufmann, Los Allos, CA, pp. 165--176, 1992.

\bibitem{Shchepin}
 E. Shchepin.
\newblock Real-valued functions and spaces close to normal.
\newblock {\em Siberian mathematical journal}, {\bf 13} (1972)
820--830


\bibitem{S}
P. Simons, PARTS. A Study in Ontology, Oxford, Clarendon Press,
1987.


\bibitem{Stell}
J. Stell, Boolean connection algebras: A new approach to the
Region Connection Calculus,{\it Artif. Intell.} \textbf{122\/}
(2000), 111--136.


\bibitem{Stone}
M. Stone.
\newblock Topological representations of distributive lattices and {B}rouwerian
  logics.
\newblock {\em {\v C}asopis P{\v e}st. Mat.}, 67, (1937), 1-25.

\bibitem{Thron}
W. J.  Thron.
\newblock {\em Proximity structures and grills},
\newblock  Math. Ann., {\bf 206\/} (1973), 35--62.

\bibitem{V}
D. Vakarelov, Region-Based Theory of Space: Algebras of Regions,
Representation Theory and Logics. In: Dov Gabbay et al. (Eds.)
{\it Mathematical Problems from Applied Logics. New Logics for the
XXIst Century.} II. Springer, 2007, 267-348.

\bibitem{vdb}
D. Vakarelov, I. D{\"u}ntsch and B. Bennett, B.
\newblock A note on proximity spaces and connection based mereology.
\newblock In Welty, C. and Smith, B., editors, {\em Proceedings of the 2nd
  International Conference on Formal Ontology in Information Systems
  (FOIS'01)}, ACM, (2001), 139-150.

\bibitem{VDDB}
D. Vakarelov, G. Dimov, I. D{\"u}ntsch, and B. Bennett.
\newblock A proximity approach to some region based theory of
space.
\newblock {\em Journal of applied non-classical logics}, vol. 12,
No3-4 (2002), 527-559

\bibitem{deVries}
 H. de Vries.
\newblock {\em Compact spaces and compactifications}, Van Gorcum,
1962

\bibitem{WHG1}
M. Winter, T. Hahmann, M. Gruninger: On the algebra of regular sets.
\newblock  Properties of representable Stonian p-ortholattices.
\emph{Annals of Mathematics and Artificial Intelligence}, 65(1):25-60, Springer, 2012.

\bibitem{WHG2}
M. Winter, T. Hahmann, M. Gruninger: On the Skeleton of Stonian p-Ortholattices.
\newblock \emph{Proc. of the 11th Int. Conference on Relational Methods in Computer Science} (RelMiCS/AKA-09), 2009. LNCS 5827, Springer, pp. 351--365.



\bibitem{W} A. N. Whitehead, \textit{Process and
Reality}, New York, MacMillan, 1929.


\end{thebibliography}
\end{document}